\documentclass[12pt,leqno,intlimits]{amsart}
\usepackage{amssymb}
\usepackage{hyperref}
\usepackage{color}

\numberwithin{equation}{section}

\newtheorem{thm}{THEOREM}[section]
\newtheorem{lem}[thm]{Lemma}
\newtheorem{cor}[thm]{Corollary}
\newtheorem{prop}[thm]{PROPOSITION}

\newtheorem{quest}[thm]{PROBLEM}

\theoremstyle{definition}
\newtheorem{defn}[thm]{Definition}%[section]
\theoremstyle{remark}
\newtheorem{rem}[thm]{Remark}
\newtheorem{ex}[thm]{Example}

\newcommand{\tref}[1]{Theorem~\ref{#1}}
\newcommand{\cref}[1]{Corollary~\ref{#1}}
\newcommand{\pref}[1]{Proposition~\ref{#1}}

\newcommand{\lref}[1]{Lemma~\ref{#1}}
\newcommand{\exref}[1]{Example~\ref{#1}}

\newcommand{\R}{\mathbb{R}}

\def\BAD{\mathop{\rm BAD}\nolimits}%

\def\:{\colon}
\newcommand*{\set}[2]{\left\{\,\left.{#1}\vphantom{#2}\,\right.;\,{#2}\,\right\}}
\def\emptyset{\varnothing}
\newcommand*{\GHto}{\mathbin{\begin{picture}(16,3.5)
\put(-1,1){$\longrightarrow$}
\put(13,-3){\llap{\text{\sf\tiny GH}}}
\end{picture}}}

\def\parit#1{\medskip\noindent{\it #1}}

\def\qeds{\qed\par\medskip}

\begin{document}
%\tableofcontents
\pagebreak
%\bibliographystyle{alpha}

%\pagenumbering{roman}

\title{Metric-measure boundary and geodesic flow on Alexandrov spaces}
\

\author{Vitali Kapovitch, Alexander Lytchak and Anton Petrunin}
%\address{Mathematisches Institut\\ Universit\"at Bonn\\
%Wegelerstrasse 10, 53115 Bonn, Germany\\}
%\email{lytchak\@@math.uni-bonn.de}

%\subjclass{53C20, 52B99}
%\footnotetext[1]{}

%\keywords{Semi-convex functions, Alexandrov spaces, differentials}
%Spherical building, Euclidean building, non-positive curvature}

%\thanks{to Professor Ballmann}

%\date{\today}
%\date{August, 1997}

\begin{abstract}
We relate the existence of many infinite geodesics on Alexandrov spaces to a statement about the average growth of volumes of balls. We deduce that the geodesic flow exists and preserves the Liouville measure in several important cases. The developed analytic tool has close ties to integral geometry.
\end{abstract}

\maketitle
\renewcommand{\theequation}{\arabic{section}.\arabic{equation}}
\pagenumbering{arabic}

%\tableofcontents

\section{Introduction}
\subsection{Motivation and application}
The following question in the theory of Alexandrov spaces was formulated in a slightly different way in \cite{PP} and remains open.
\begin{itemize}
\item Are there ``many'' infinite geodesics on any Alexandrov space without boundary?
\end{itemize}

We address this question and obtain an affirmative answer in several cases.
 The main new tool is the investigation of  the Taylor expansion of the average volume growth.
The central results relate the first coefficient of this expansion to  the geodesic flow and show how to control
the Taylor expansion. This tool  might be interesting in its own right, beyond the realm of Alexandrov geometry.

In particular, we prove the existence of such infinite geodesics in the most classical examples of non-smooth Alexandrov spaces:

\begin{thm} \label{thmfirst}
Let $X$ be the boundary of a convex body in $\mathbb R^{n+1}$.
Then almost any direction in the tangent bundle $TX$
is the starting direction of a unique infinite geodesic on $X$.
Moreover, the geodesic flow is defined almost everywhere and preserves the Liouville measure.
\end{thm}

Apparently, the existence of a \emph{single} infinite geodesic has not been known, even in the two-dimensional case \cite{Zam-quest}.
Our result might appear somewhat
 surprising
since on \emph{most} convex surfaces \emph{most points in the sense of Baire categories} are not inner points of any geodesic;
see \cite{Zam-inv}.

\subsection{Metric-measure-boundary}
On a smooth manifold with boundary the geodesic flow is not defined for all times. The amount of geodesics terminating at the boundary in a given time depends on the the size of this boundary,  due to Santalo's integral  formula.

 We are going to capture the size of the boundary by estimating the average volumes of small balls and their deviations from the corresponding volumes in the Euclidean space.

Let $(X,d)$ be a locally compact separable metric space,
$\mu$ be a Radon measure on $X$ which takes finite values on the bounded subsets.
%and $n$ be a natural number.%???Maybe we should say that X is proper???
%(The definition will be applied in the case when $n$ is the Hausdorff dimension of $(X,d)$.)
{For $x\in X$ and $r>0$
denote by $B(x,r)$ the open metric ball of radius $r$ around $x$. Consider the  volume growth function}  $b_r\:X\to [0,\infty )$,
\begin{equation}
b_r(x):=\mu (B(x,r)).
\end{equation}
{ For a natural number $n>0$,  let $\omega _n$ be the volume of the $n$-dimensional unit Euclidean ball. The  deviation function
$$v_r (x)=1 - \frac {b_r (x)} { \omega _n{\cdot}r^n} $$
measures in a very rough sense the deviation of the metric measure space $(X,d,\mu)$ from $\R^n$.  Moreover, one can expect the behaviour of $v_r$ at the origin $r=0$ to reflect some curvature-like properties of the space $X$, as in the following fundamental example.

\begin{ex} \label{smoothscal}
Let $X^n$ be a smooth Riemannian manifold with Riemannian volume $\mu$. Then, $v_r (x)= \frac 1 {6\cdot(n+2)}\cdot scal \cdot r^2$ up to terms of higher order in $r$.  Here $scal$ denotes the scalar curvature of $X$.
%equivalently,
%\[\int_A b_r\cdot d\mu=\omega_n\cdot \mu A\cdot r^n+\frac{(scal\cdot \mu)A}{6 \cdot(n+2)}\cdot r^{n+2}\]
%for any measurable set $A\subset X$
\end{ex}

In this paper we are interested in the order of vanishing of $v_r$ at $r=0$  and the first non-vanishing coefficient;
in particular we assume that $v_r$ converges to zero in some integral sense.
%The question can only be meaningful under  some minimal regularity assumptions on the metric measure space $(X,d,\mu)$.
In most interesting metric spaces $(X,d)$, at least in the cases investigated here, the only reasonable choice of the measure $\mu$ for which $v_r$ is ``sufficiently small'' in $r$ is the $n$-dimensional Hausdorff measure $\mathcal H^n$:

\begin{ex}      \label{ex:rect}
Let $X$ be  a countably $n$-rectifiable metric space.
Assume that the Radon measure $\mu$  non-zero on open subsets of $X$.
If $\mu =\mathcal H^n$ then the functions $v_r$ converge $\mathcal H^n$-almost everywhere to $0$. Moreover, $\mu=\mathcal H^n$  is the only measure  with this property~\cite{AmbKirk}[Theorem 5.4].
 \end{ex}

 Therefore, in the sequel, the number $n$ will always be the Hausdorff dimension of $X$ and $\mu$ will be the $n$-dimensional Hausdorff measure.

As seen in Example \ref{ex:rect}, most points in reasonably nice spaces are rather regular.
It is  conceivable, that by averaging  the deviation functions $v_r$ we will smooth out the ``wildest singularities''.
The obtained objects will experience  better behaviour at $r=0$ and tell us more about the  regularity of the space.

%\begin{ex} If the functions $v_r$  for $0<r\leq 1$ are locally  uniformly bounded on $X$ then the measure $\mu$ is absolutely continuous with respect %to the $n$-dimensional Hausdorff measure $\mathcal H^n$.  If the space $X$ is $n$-dimensional rectifiable  then for $\mu =\mathcal H^n$ the deviation
%functions $v_r$ converge almost everywhere to $0$ for $r\to 0$.
% always normalized so that it coincides with the Lebesgue measure on $\R^n$.}
%\end{ex}

Thus, instead of looking on the point-wise behaviour of $v_r$ at $r=0$ we
 define the \emph{deviation measure} $\mathcal{V}_r$ of $X$ as a signed Radon measure \begin{equation} \label{eq:first}
\mathcal{V}_r = v_r \cdot \mu \, ,
\end{equation}
absolutely continuous with respect to $\mu$.}
%Here $\omega _n$ is the volume of the $n$-dimensional unit Euclidean ball.

The vector space  $\mathrm M(X)$ of \emph{signed Radon measures on $X$ is   dual  to the topological  vector} space of compactly supported continuous functions $C_c (X)$.
We consider the space $\mathrm M(X)$ with the topology of weak convergence.
Recall that a subset $\mathcal F \subset \mathrm M(X)$ is relatively compact if and only if
it is \emph{uniformly bounded};
that is, if for any compact subset $K\subset X$ the values $\nu (K), \nu \in \mathcal F$ are uniformly bounded.

%In all examples and applications below, $\mu$ is the $n$-dimensional Hausdorff measure $\mathcal H^n =\mathcal H^n _X$ on $X$, always normalized so that it coincides with the Lebesgue measure on $\R^n$.

%The map $r\mapsto \mathcal{V}_r \in \mathrm M(X)$ measures the deviation of $X$ from $\R^n$.
%The Taylor coefficients of $\mathcal{V}_r$ at $0$ (if they are defined) can be interpreted as integral curvatures, in analogy with the following exam

The next example, fundamental for this paper,  can be obtained by computations in local coordinates.
% and together with Example \ref{smoothscal} is the
%motivating example for this paper.
Since it is formally not needed in the sequel, we omit
the details, however  a rigorous  proof can be extracted from the proof of \tref{alexandrovthm} in Section \ref{sec:Alex} below.

\begin{ex} \label{mainex}
Let $X$ be a smooth $n$-dimensional Riemannian manifold with boundary $\partial X$.
Then, for $r\to 0$, the measures $\mathcal{V}_r/r$ converge in $\mathrm M(X)$ to
$c_n \cdot \mathcal H^{n-1} _{\partial X}$, for some constant $c_n >0$ depending only on $n$.
\end{ex}

This example suggests to view the first Taylor coefficient of $\mathcal{V}_r$ as the ``boundary'' of the metric-measure space $(X, d, \mu)$.
It motivates the following definition.

\begin{defn} \label{def:first}
Let $(X,d,\mu)$ be a metric measure space as above. Let $\mathcal{V}_r$ be the deviation measure of $X$, as in \eqref{eq:first}.
%and let the Radon measure $|\mathcal{V}_r|$ be the total variation of $\mathcal{V}_r$.
We say that $X$ has locally finite metric-measure boundary, abbreviated as mm-boundary,
if the family of signed Radon measures
\[\set{\mathcal{V}_r/r }{0<r\leq 1}\]
is uniformly bounded.
% we say that $X$ has locally finite mm-boundary ({\color{red} or metric-measure boundary}).
If $\lim _{r\to 0} \mathcal{V}_r /r =\nu $ in $\mathrm M(X)$, we call $\nu$ the mm-boundary of $X$. If $\nu =0$ we say that $X$ has vanishing mm-boundary.
\end{defn}

We refer to Subsection \ref{subsec:example} and Section \ref{sec:final} for a discussion of examples and questions, and
state now our central result connecting mm-boundaries to the existence of infinite geodesics in Alexandrov spaces:

\begin{thm} \label{thmmain}
Let $X$ be an Alexandrov space.
If $X$ has vanishing mm-boundary, then almost each direction of the tangent
bundle $TX$ is the starting direction of an infinite geodesic.
Moreover, the geodesic flow preserves the Liouville measure on $TX$.
\end{thm}

In different settings, geodesic flow on singular spaces have been investigated in \cite{BallmannBrin} and \cite{Bamler}.

\subsection{Size of the mm-boundary in Alexandrov spaces}
The next theorem shows that, similarly to \exref{mainex}, the topological boundary is closely related to the mm-boundary in Alexandrov spaces.

\begin{thm} \label{alexandrovthm}
Let $X^n$ be an $n$-dimensional Alexandrov space.
Then $X$ has locally finite mm-boundary.
If $ \nu= \lim \frac{\mathcal{V}_{s_j}}{s_j}$, for a sequence $s_j \to 0 $, then $\nu$ is a Radon measure and
the following holds true.
\begin{enumerate}
\item\label{full-measure-zero-nu}There is a Borel set $A_0$ with $\mathcal H^n (X\setminus A_0) = \nu (A_0)=0$.
\item\label{bry-nu} If the topological boundary $\partial X$ is non-empty then $\nu \geq c \cdot \mathcal H^{n-1} _{\partial X}$,
for a positive constant $c$ depending only on $n$.
\item \label{n-1-nu} If the topological boundary $\partial X$ is empty then $\nu (A)=0$, for any Borel subset $A\subset X$ with $\mathcal H^{n-1} (A)<\infty$.
\end{enumerate}
\end{thm}

We believe that an Alexandrov space with empty topological boundary $\partial X$ has vanishing mm-boundary,
which would solve the question about the existence of infinite geodesics.
This conjecture will be proved in two cases.

\begin{thm} \label{hypersurface}
Let $X^n$ be a convex hypersurface in $\R^{n+1}$ or let $X$ be a two-dimensional Alexandrov space without boundary.
Then $X$ has vanishing  mm-boundary.
\end{thm}

In combination with \tref{thmmain} this proves \tref{thmfirst}.
The two-dimensional case  could be derived  from the statement about convex hypersurfaces and Alexandrov's embedding theorems.
Another proof follows from a much stronger result discussed in the next subsection.

\subsection{Metric-measure-curvature} \label{subsec:curv}
Motivated by Example \ref{smoothscal} one can naively hope that the second Taylor coefficient at $0$ of the
map $\mathcal{V}_r\colon r\mapsto \mathrm M(X)$ describes the scalar curvature of the space.
%This again has natural analogy with the classical integral geometry.

\begin{defn}
Let $X,\mathcal{V}_r$ be as in Definition \ref{def:first}.
If the family $\mathcal{V}_r /r^2, r\leq 1$ is uniformly bounded then we say that $X$ has locally finite mm-curvature.
% For any compact subspace $K$ of $X$, we call $\limsup _{r\to 0} \|\mathcal{V}_r /r\| _K$ the bound of the scalar curvature in $K$.
If the measures $\mathcal{V}_r /r^2$ converge to a measure $\nu$, we call $\nu$ the mm-curvature of $X$.
\end{defn}

%If, in the definition above, the measures $\mathcal{V}_r /r^2$ locally converge to a measure $\nu$, we call $\nu$ the mm-curvature of $X$.
Clearly, local finiteness of mm-curvature as defined above implies that the mm-boundary vanishes.
Thus, the following result proves \tref{hypersurface} in the $2$-dimensional case.

\begin{thm} \label{intsurface}
Let $X$ be a 2-dimensional Alexandrov space without boundary.
Then $X$ has locally finite mm-curvature.
\end{thm}

This finiteness result holds true in the much greater generality of surfaces with bounded integral curvature in the sense of Alexandrov--Zallgaler--Reshetnyak \cite{Reshetnyak-GeomIV}, \cite{AZ}, see Section \ref{sec:surface}.

Note, however, that the mm-curvature
in \tref{intsurface} does not need to coincide with the ``curvature measure'' as defined in \cite{AZ}, even in the case of a cone; compare to Example \ref{ex:cone}.
In particular, this shows that the mm-curvatures in $2$-dimensional Alexandrov spaces are not stable under Gromov--Hausdorff convergence.

\begin{rem}
Nina Lebedeva and the third named author have found in \cite{LP} a ``scalar curvature measure'' on all smoothable Alexandrov spaces.
% which has the property of being stable
%under non-collapsed Hausdroff convergence.
There is a hope, supported by our proof of \tref{intsurface}, that a better understanding of this ``stable curvature measure'' will
lead to some control of the mm-boundary and mm-curvature discussed here.
\end{rem}
\subsection{Relation to the Lipschitz--Killing curvatures} Let $M$ be 
%compact convex body or 
a compact smooth submanifold in $\R^n$.
Given $r>0$, consider the volume $w(r)=\mathcal H^n (B(M,r))$ of the distance tube $B(M,r)$ around $M$.
The function $r\mapsto w(r)$ is a polynomial, at least for small positive $r$. The coefficients of $w(r)$,  called the  Lipschitz--Killing curvatures of $M$, are given as integrals of some intrinsically defined curvature terms.
Moreover, these coefficients can be localized and considered as measures on $M$.
We refer to \cite{Alesker} for a short account of the theory, connection of the theory with
\cite{LP} and further hypothetical relations with the theory of Alexandrov spaces.

To make the formal similarity with our approach to mm-boundary and mm-curvature more transparent, we observe that
(at least for a smooth $n$-dimensional manifold $M$) the number $\int _M \mathcal H^n (B(x,r))\, \cdot d\mathcal H^n (x)$ can be interpreted as the $\mathcal H^{2\cdot n}$-measures
of the distance tubes $B(\Delta,{\frac r {\sqrt 2} })$ around the diagonal $\Delta $ in the Cartesian product $M\times M$.

 \subsection{Idea of the proof of \tref{thmmain}}
The interpretation of the tangent bundle of $M$ as the normal bundle of the diagonal $\Delta$ in $M\times M$ gives a connection between
the measure theoretical properties of the tubes around $\Delta$ and the dynamical properties of the geodesic flow.

We clarify this abstract statement by explaining the main idea of our proof of \tref{thmmain} in the case of a complete smooth Riemannian manifold $X=M$.
In this case the existence of geodesics is trivial.
Thus, we just sketch a new proof of the classical fact that the geodesic flow $\phi$ preserves the Liouville measure $\mathcal M$ on $TM$.
This proof is sufficiently stable to be transferred to the singular situation,

Denote by $\pi\: TM\to M$ the tangent bundle of $M$.
Let $\phi_t\: TM\to TM$ be the geodesic flow for time $t$.
 Define $E\:TM\to M\times M$ by
\[E(v)=(\pi(v), \pi (\phi _1 (v)).\]
By construction, $E(-\phi_1(v))= J(E(v))$,
where $J$ is the involution of $M\times M$ which switches the coordinates.
Since $J$ preserves the measure $\mathcal H^{2\cdot n}$ on $M\times M$ and $v\to -v$ preserves the Liouville measure $\mathcal M$ on $TM$,
the statement that $\phi$ is measure preserving hinges upon the smallness of measure-distortion of the map $E \:(TM,\mathcal M) \to (M\times M, \mathcal H^{2\cdot n})$ close to the $0$-section.

In the present case of a Riemannian manifold,
this property of $\phi_1$ is expressed by the fact that the differential of $E$ is the identity (after suitable identifications).
Similarly,
in the general case of Alexandrov spaces, we observe that the ``infinitesimal'' deviation (via the canonical map $E$) between $J$ being measure preserving (which we know) and $\phi _1$ being measure preserving (which is what we want to show)
is expressed as the triviality of the mm-boundary.

\subsection{Stability and relation with quasi-geodesics}
Many Alexandrov spaces, for instance all convex hypersurfaces, appear naturally as Gromov--Hausdorff limits of smooth Riemannian manifolds. However, the properties of the geodesic flow, mm-boundaries and mm-curvature are unstable under limit operations; see also the discussion at the end of Subsection \ref{subsec:curv}.
Thus, there is no hope to deduce \tref{intsurface}, \tref{hypersurface} or \tref{thmfirst} by a direct limiting argument.

For instance, being a geodesic is a local notion, not preserved under limits.
However, any limit of geodesics in a non-collapsed limit of Alexandrov spaces is a curve sharing many properties with geodesics.
These properties are used to define the so called \emph{quasi-geodesics};
see \cite{PP}, \cite{Petsemi} and the references therein.
It was shown that any direction is the starting direction of an infinite quasi-geodesic.
One motivation for the present paper was an attempt to prove Liouville's theorem for the ``quasi-geodesic flow'', see Subsection \ref{subsec:quasi}.

\subsection{Examples} \label{subsec:example}
The estimates of the mm-boundary and mm-curvature are quite involved even in quite simple situations.
The following examples are not needed in the sequel and we omit the somewhat tedious computations. Examples \ref{ex:cone}, \ref{ex:secondlast} and \ref{ex:last} should be compared with
\cite{Bernig-CAT} and \cite{Bernig-Alex} revealing further natural connections to the theory of Lipschitz--Killing curvature on singular subsets of the Euclidean space.

%\begin{ex} Let $X$ be an $n$-dimensional \emph{Lipschitz manifold}; that is, an $n$-dimensional manifold equipped with intrinsic metric such that any point admits a neighborhood which is bi-Lipschitz to an open set in a Euclidean space.
%Let us equip $X$ with the measure $\mu=\mathcal H^n$.
%Then $\lim _{r\to 0} \mathcal{V}_r =0$. Moreover, $\mathcal H^n$ is the only measure with this property.
%\end{ex}

\begin{ex}
Let $X$ be a Riemannian manifold with a Lipschitz continuous metric.
Then $X$ has vanishing  mm-boundary.
\end{ex}

\begin{ex}
If $X$ is a manifold with two-sided bounded curvature in the sense of Alexandrov then its mm-curvature is well-defined and absolutely continuous with respect to the Hausdorff measure.
\end{ex}

\begin{ex} \label{ex:cone}

Let $X$ be the Euclidean  cone over the circle $S_{\rho}$ of length $\rho$.
The curvature measure and the mm-curvature are Dirac measures
concentrated   at the tip of the cone.  The mass of the curvature measure is $\alpha =2\pi-\rho$.
From example \ref{smoothscal} one would expect the mass of the mm-curvature to be $m(\alpha) = \frac {\alpha} {12}$.
However, a straightforward calculation shows that $m(\alpha)=\frac \alpha {12}+ f(\alpha)$, where $f(\alpha)= O(\alpha^2)$ is a non-zero function.
% $f(\alpha)= O(\alpha ^3)$.

\end{ex}

\begin{ex} \label{ex:secondlast}
Let $X$ be a finite $n$-dimensional simplicial complex with an intrinsic metric $d$.
Assume that the restriction of $d$ to each simplex is given by a smooth Riemannian metric.
Then $X$ has a finite mm-boundary $\nu$ with the support on the $(n-1)$-skeleton $X^{n-1}$.
\end{ex}

\begin{ex} \label{ex:last}
Assume that $X$ as in the last example is a pseudo-manifold.
Then $X$ has finite mm-curvature. If all simplices are flat then the mm-curvature is concentrated 
 on  the $(n-2)$-skeleton.
\end{ex}

\subsection{Structure of the paper}
After preliminaries collected in Section \ref{sec:prelim}, we prove \tref{thmmain} in Section \ref{sec:Liou} along the lines sketched above.
In Sections \ref{sec:surface}, \ref{sec:hyper} and \ref{sec:Alex} we prove the remaining theorems \ref{intsurface}, \ref{hypersurface} and \ref{alexandrovthm} respectively.
The proofs of these theorems all rely on a decomposition of the space into a regular and a singular part, with a quantitative estimate of the size of the singular part.
Finally, on the regular part we estimate the mm-curvature and mm-boundary  by comparing it to other natural measures on these spaces.

In the case of surfaces, this comparison measure is the classical curvature measure, in the case of convex hypersurfaces, this comparison measure is the mean curvature.
Finally, in the case of a general Alexandrov space, the comparison is given by the derivative of the metric tensor expressed in DC-coordinates, \cite{Per-DC}.

The needed control of the ball growth in terms of these measures is given by a theorem of Mario  Bonk and Urs Lang in the case of surfaces and follows from classical convex geometry in the case of hypersurfaces.
The analytical comparison result needed for Alexandrov spaces is established in Section \ref{sec-BV-estimate}.

In the final Section \ref{sec:final} we collect a number of comments and open questions which naturally arose during the work on this paper.

\subsection{Acknowledgments} The authors are grateful for helpful conversations and comments to Semyon Alesker,  Richard Bamler,
Andreas Bernig, Ivan Izmestiev, Aaron Naber, Koichi Nagano.

The first author was supported in part by a Discovery grant from NSERC and by a Simons Fellowship from the Simons foundation (award 390117).
The second author was supported in part by the DFG grants SFB TRR 191 and SPP 2026.
The third author was partially supported by NSF grant DMS 1309340.

\section{Preliminaries} \label{sec:prelim}
\subsection{Metric spaces}
We refer to \cite{BBI01} for basics on metric spaces.
The distance between points $x,y$ in a metric space $X$ will be denoted by $d(x,y)$.
% or $|xy|_X$ or just $|xy|$.
By $B(x,r)$ we will denote the open metric ball of radius $r$ around a point $x$. For $A\subset X$
we denote by $B (A,r)$ the open $r$-neighborhood $B (A,r) =\cup _{x\in A} B (x,r)$.

A \emph {minimizing geodesic} $\gamma$ in a metric space $X$ is a map $\gamma \: \mathbb I\to X$ defined on an interval $\mathbb I$ such that for some number $\lambda \geq 0$ and all $t,s\in \mathbb I$
$$ d(\gamma (t),\gamma (s)) =\lambda \cdot |t-s|.$$
In particular, we allow $\gamma$ to have any constant velocity $\lambda \geq 0$.
% On the other hand, geodesics
%as defined here will always be globally ``minimizing geodesics'' in the usual notations of Riemannian geometry.
A \emph{geodesic} is a curve $\gamma\: \mathbb I\to X$ such that its restriction to a small neighborhood of any point in $\mathbb I$ is a minimizing geodesic.
Note that a geodesic is a curve of constant velocity.

\subsection{Metric measure spaces}
We refer to \cite{Federer} and \cite{Evans} for basics on measure theory.

Let $X$ be a locally compact separable metric space.
A \emph{Radon measure} on $X$ is a measure on $X$ for which all compact subsets are measurable and have finite measure. 
 Any Radon measure defines an element of  $\mathrm M(X)$ the dual space to the topological vector space $\mathcal C_c (X)$ of compactly supported continuous functions on $X$.
All elements in $\mathrm M(X)$ are called \emph{signed Radon measures}.
 Any  $\mu \in M(X)$ can be uniquely written as $\mu_+-\mu_-$, where $\mu_\pm$ are a Radon measures  concentrated on disjoint subsets.
  The  measure $|\mu|=\mu_+ +\mu_-$ is called \emph{the total variation} of $\mu$.

%Any  $\mu \in M(X)$ can be uniquely written as $f\cdot |\mu|$, where $|\mu|$ is a Radon measure, called \emph{the
%total variation} of $\mu$, and $f\in L^1 (X,\mu)$ takes only the values $\pm 1$.   }

A family $\mathcal F$ of signed Radon measures on $X$ is \emph{uniformly bounded} if for any compact subset $K\subset X$ there exists a constant $C(K)>0$ such that $|\mu| (K) \leq C(K)$ for any $\mu \in \mathcal F$.
Any uniformly bounded sequence of signed measures $\mu _i$ has a convergent subsequence.

The following lemma will be repeatedly used.

\begin{lem} \label{lem:exchange}
Let $X$ be a metric space with two Radon measures $\mu $ and $\nu$. Let $r>0$ be arbitrary and let $A\subset X$ be a Borel subset. Then
$$\int _A \mu (B(x,r)) \cdot d\nu (x) \leq \int _{B (A,r)} \nu (B (x,r)) \cdot d\mu (x).$$
\end{lem}

\begin{proof} Due to Fubini's theorem
the left hand side is the volume of
$$S=\set{(y,x) \in X\times X}{y \in A,\, d(y,x) <r}$$
with respect to the product measure $\nu \otimes \mu$. %???why not \oplus???
And on the right hand side of the inequality is the volume of the larger set
$$T=\set{(y,x) \in X\times X}{x\in B(A,r), \, d(y,x)<r}$$
with respect to the same measure.

Since $S\subset T$, the statement follows.
\end{proof}

\subsection{Alexandrov spaces} \label{subsec:Alex}
We are assuming that the reader is familiar with basic theory of Alexandrov spaces and refer to~\cite{BGP} as an introduction to the subject.
In this paper, an \emph{Alexandrov space} is a complete, locally compact, geodesic metric space of finite Hausdorff dimension and of curvature bounded from below by some $\kappa \in \R$.
For Alexandrov spaces, an upper index will indicate the Hausdorff dimension; that is, $X^n$ denotes an $n$-dimensional Alexandrov space, equipped with the $n$-dimensional Hausdorff measure $\mathcal H^n$.

The set of starting directions of geodesics starting at a given point $x\in X$ carries a natural metric, whose completion is the \emph{tangent space} $T_x=T_xX$
of $X$ at the point $x$. It is  an $n$-dimensional Alexandrov space of non-negative curvature. Moreover, it is the Euclidean cone
over the space $\Sigma _x$ of unit directions. The Euclidean cone structure defines multiplications by positive scalars $\lambda \geq 0$ on $T_xX$. The origin
of the cone $T_xX$ is denoted by $0=0_x$.
Elements of $T_xX$ are called \emph{tangent vectors} at $x$, despite that $T_xX$ is not a vector space in general.
For $v\in T_xX$ the \emph{norm} $|v|$ of $v$ is the distance of $v$ from the origin $0_x$.

Geodesics in $X$ do not branch,  moreover, any two geodesics with  identical starting vectors coincide.
For $x\in X$ the \emph{exponential map} $\exp_x$ is defined as follows.
Let $D_x$ denote the set of all vectors $v\in T_xX$ for which there exists an (always unique) \emph{minimizing geodesic} $\gamma_v\:[0,1] \to X$ with starting direction $v$. The exponential map is defined on $D_x$ as
$$\exp _x (v)= \gamma _v(1).$$
For any $r>0$, the map $\exp_x$ sends $D_x\cap B (0_x,r)\subset T_xX$ surjectively onto $B(x,r)\subset X$.
Moreover, for a constant $C=C(\kappa)\geq 0$ and all $r< \frac 1 C$ the map $\exp_x \:D_x \cap B(0_x,r) \to B(x,r)$ is $(1+ C \cdot r^2)$-Lipschitz continuous.

By the theorem of Bishop--Gromov, the volume $b_r (x)=\mathcal H^n (B(x,r))$ is bounded from above by the corresponding volume in the
space of constant curvature $\kappa$. In particular, $b_r (x) \leq \omega _n \cdot r^n + C \cdot r^{n+2}$ for all $r\leq \frac 1 C$, where
the constant $C$ can be chosen as before. Thus, the deviation measures $\mathcal{V}_r $ from \eqref{eq:first} satisfy
$$\mathcal{V}_r \geq -C\cdot r^2\cdot  \mathcal H^n,$$
for all sufficiently small $r$.
Here and in the previous paragraph, one can set $C=0$ if $\kappa \geq 0$.

Denote by $X_{reg}$ the set of all points $x\in X$ with $T_xX$ isometric to the Euclidean space.
The set $X_{reg}$ has full $\mathcal H^n$-measure in $X$.
Any inner point of any geodesic starting on $X_{reg}$ is contained in $X_{reg}$, \cite{Petparallel}.

The topological boundary $\partial X$ of $X$ can be defined as the closure of the set of all points $x\in X$ with $T_xX$ isometric to a Euclidean half-space. Up to a subset of Hausdorff dimension $n-2$, $\partial X$ is an $(n-1)$-dimensional Lipschitz manifold.

\subsection{Volume and bi-Lipschitz maps}
Let $\mu =\mathcal H^n$ be a Radon measure on the metric space $X$.
Let $U\subset X$ and $V\subset \R^n$ be open and assume that there is an $(1+\delta)$-bi-Lipschitz map $f\:U \to V$;
that is,
\[\frac1{1+\delta}\le\frac{|f(x)-f(y)|}{d(x,y)}\le 1+\delta\]
for any pair of distinct points $x,y\in U$.

Let $A\subset U$ be given with
$B (A,{(1+\delta) \cdot r}) \subset U$ and $B (f(A),{(1+\delta) \cdot r}) \subset V$. Then, for all $x\in A$,
\begin{equation} \label{eq:bilip}
(1+\delta) ^{-2\cdot n} \leq \frac {b_r(x)} {\omega _n \cdot r^n} \leq (1+\delta ) ^{2\cdot n}.
\end{equation}
Therefore, if $\delta $ is sufficiently small, $|\mathcal{V}_r | (A) \leq 3\cdot  n \cdot \delta \cdot \mathcal H^n (A).$

\section{Liouville measure and geodesics} \label{sec:Liou}
\subsection{Tangent bundle and Liouville measure} \label{subsec:tb}
Let $X$ be an $n$-dimensional Alexandrov space.
Denote by $TX$ the disjoint union of the tangent spaces at all points,
\[TX=\bigsqcup_{x\in X} T_x X.\]
Let $\pi\:TX\to X$ be the footpoint projection, so $\pi (T_xX)=\{x\}$ for any $x\in X$. %pi --> tau???
For a subset $K\subset X$ denote by $TK$ the inverse image $\pi^{-1} (K)= \cup _{x\in K} T_xX$. %???T_K???
Given $r>0$, denote by $T^r K$ the set of all vectors in $TK$ of norm smaller than $r$.

The Riemannian structure on the set of regular points discussed in \cite{Otsu-Shioya} (see also \cite{Shioya}, \cite{Per-DC}) provides $TX_{reg}$ with a structure of a Euclidean vector bundle over $X_{reg}$. In this topology, for any sequence of geodesics $\gamma _i$ in $X_{reg}$ converging to a geodesic $\gamma$, the starting directions of $\gamma _i$ converge to the starting direction of $\gamma$.

On the Euclidean vector bundle $TX_{reg}$ over $X_{reg}$ we have a natural choice of measure, which locally coincides with the product measure of $\mathcal H^n _X$ and the Lebesgue measures on the fibers. More precisely,
it is the unique Borel measure $\mathcal M$ on $TX_{reg}$ such that for any Borel set $A\subset TX_{reg}$
$$\mathcal M(A)= \int _X \mathcal H^n(A \cap T_x X) \cdot d\mathcal H^n (x).$$
We extend $\mathcal M$ to a measure on $TX$ by setting $\mathcal M(TX\setminus TX_{reg})$ to be $0$.

By the definition, a subset $A\subset TX$ is $\mathcal M$-measurable if and only if there exists a Borel subset $A'\subset A\cap TX_{reg}$ such that
for $\mathcal H^n$-almost all $x\in X$ the intersection $(A\setminus A') \cap T_xX$ has $\mathcal H^n$-measure zero in $T_xX$.

For any $\lambda>0$, we have $\mathcal M(\lambda\cdot A)=\lambda^n\cdot\mathcal M (A)$ for any measurable set $A\subset TX$.
The involution $I\:TX_{reg}\to TX_{reg}$, defined by $I(v)=-v$,
preserves $\mathcal M$
since it preserves the Lebesgue measure in each tangent space.

\subsection{Geodesic flow}
Let us define the geodesic flow $\phi$ on a maximal subset $\mathcal F$ of $TX\times\R$.

For any $v\in T_xX$ we set $\phi _0(v)=v$.
If no geodesic starts in the direction of $v$,
the value $\phi _t(v)$ will not be defined for $t\neq 0$. If such a geodesic $\gamma_v$ exists, then $\gamma_v$ can be uniquely extended to a maximal possible half-open interval $\gamma_v \:[0,a)\to X$. For $t\geq a$
the value $\phi _t(v)$ will not be defined. For $0<t<a$ we set $\phi _t(v)$ to be $\gamma _v ^+ (t) \in T_{\gamma _v(t)} X$, the starting direction of $\gamma_v\:[t,a) \to X$ at $\gamma_v (t)$.

If the geodesic $\gamma _v\:[0,a)\to X$ extends to an (again uniquely defined, maximal) geodesic $\gamma_v\:(b,a) \to X$ for some $b<0$
then we define $\phi _t(v)$ for $b<t<0$ to be $\gamma_v ^+ (t) $ as above.

We denote by $\mathcal F$ the set of all pairs  $(v,t)\in TX\times\R$  for which $\phi_t(v)$ is defined.

For $\lambda >0$, for $t,s\in \R$ and $v\in T_x X$ we have
$$\phi _t(\lambda \cdot v) =\lambda \cdot \phi _{\lambda \cdot t} (v)
\quad\text{and}\quad
\phi _{t+s} (v) =\phi _t ( \phi _s (v)),$$
whenever the right hand side is defined.

The partial flow $\phi$ preserves the norm of tangent vectors. Since inner points of geodesics starting in $X_{reg}$ are contained in $X_{reg}$, the set $TX_{reg}$ is invariant under the flow $\phi$.
%On $TX_{reg}$ we have
%$I\circ \phi_t = \phi _{-t} \circ I$, whenever both sides are defined.

By construction, the domain of the definition of the geodesic flow almost includes the domain of the definition of the exponential map.
More precisely, consider the set
\[D=\bigcup_{x\in X} D_x\subset TX;\]
that is, the set of all vectors $v\in TX$ for which $\exp _{\pi(v)} (v)$ is defined.
Note that $\lambda\cdot D\subset D$ for any $0\leq \lambda \leq 1$.
Moreover, for all $v\in D$ and all $0\leq \lambda <1$
the geodesic flow $\phi_1(\lambda\cdot v)$ is defined (equivalently $(\lambda\cdot v,1)\in \mathcal F$) and
\[\pi(\phi_1 (\lambda\cdot  v))= \exp _{\pi (v)} (\lambda\cdot  v).\]

Thus, for $\mathcal M$-almost all $v\in D$ we have the following
\begin{itemize}
\item $v\in TX_{reg}$;
\item $\phi _1(v) \in TX_{reg}$ is defined,  hence $(v,1)\in \mathcal F$;
\item $w=-\phi_1 (v)\in D$ and
\begin{equation} \label{eq:symm}
(\pi (w), \exp (w))=(\exp (v), \pi (v)) \in X\times X.
\end{equation}
\end{itemize}

\subsection{Measurability} \label{subsec:measur}
In order to use measure theoretic arguments we will need the following lemma, see also Subsection \ref{subsec:quasi}.
\begin{lem}\label{lem:measurability}
The set $\mathcal F\subset TX\times \R$ is measurable with respect to the product of the Liouville's measure $\mathcal M$ on $TX$ and the Lebesgue measure on $\R$.
Moreover the map $\phi \:\mathcal F \to TX$ is measurable.
\end{lem}

\begin{proof}
Fix $(v,\tau)\in \mathcal F$ and set $\gamma(\tau\cdot t)=\phi_t(v)$, $t\in[0,1]$.
Note that there exists some $k>0$ such that the restriction of $\gamma$ to any subinterval of length $\frac 1 k$ is a (minimizing) geodesic.
 We will call such $\gamma$ a  $k$-geodesic and  write $(v,\tau)\in \mathcal F_k$.

 The limit of any converging sequence of $k$-geodesics is a $k$-geodesic.
Hence $\mathcal F_k'=\mathcal F_k\cap (TX_{reg} \times \R)$ is a closed set in $TX_{reg} \times \R$.
 Therefore, $\mathcal F\cap (TX_{reg} \times \R)$  is a countable union of closed subsets $\mathcal F_k'$, hence measurable.
Moreover, the restriction $\phi\: \mathcal F_k' \to TX_{reg}$ is continuous and, therefore,
$\phi \:\mathcal F\cap (TX_{reg} \times \R) \to TX_{reg}$ is a Borel-measurable map.

Since $\mathcal M(X\backslash X_{reg})=0$ the statement follows.
\end{proof}

\subsection{Liouville property}
Denote by $\mathcal G$ the set of all vectors $v\in TX $ such that $\phi _t(v)$ is defined for all $t\in \R$.
Note that $\mathcal G$ contains the $0$-section;
it is invariant under multiplications by any $\lambda>0$
and it is invariant under the geodesic flow $\phi$.
Moreover, $\mathcal G\cap TX_{reg}$ is invariant under the involution $I (v)=-v$.

\begin{defn}
We say that an Alexandrov space $X$ has the \emph{Liouville property} if $\mathcal M(TX\setminus \mathcal G)=0$ and
for any $t\in \R$ the geodesic flow $\phi _t \:\mathcal G\to \mathcal G$ preserves the Liouville measure.
\end{defn}

The Liouville property can be checked infinitesimally using the following lemma.

\begin{lem} \label{infini}
An Alexandrov space $X$ \emph{does not} have the Liouville property if and only if there is a compact subset $K\subset X$, a positive number
$\varepsilon$ and a sequence of positive numbers $r_m \to 0$ with the following property.
For every $m$, there exists a Borel subset $A_m\subset T^{r_ m} K$ such that
\begin{equation} \label{eq:m}
\varepsilon \cdot r_m^{n +1} \leq \mathcal M (A_m) -\mathcal M (\phi _1 (A_m)).
\end{equation}
 Here $\phi_1 (A_m)$ is the set of all $\phi _1(v), v\in A_m$, for which
$\phi _1(v)$ is defined.
\end{lem}

\begin{proof}
If at least one $r_m$ with the above property exists, then $X$ does not have the Liouville property by definition.

Assume that $X$ does not have the Liouville property.
Then, by homogeneity of the geodesic flow, $\phi _1$ is either undefined  on a subset of $TX$ with positive measure
or it does not preserve the measure $\mathcal M$.
In both cases we can find a compact subset $K_1\subset X_{reg}$, a Borel subset $A\subset T^1 K_1$  and $\varepsilon >0$ such that
\[\varepsilon < \mathcal M (A) - \mathcal M(\phi_1 (A)).\]

Since
\begin{align*}
\phi_1 (A)&=2\cdot \phi _2 (\tfrac 1 2 \cdot A )=
\\
&=2\cdot \phi _1 \circ \phi_1 (\tfrac 1 2\cdot  A ) ,
\end{align*}
we deduce
\begin{align*}
\frac{\varepsilon}{2^n} &\leq \mathcal M \big(\tfrac 1 2\cdot A \big) - \mathcal M \big( \phi_1 (\tfrac 1 2\cdot A ) \big) +
\\
&\quad+\mathcal M \big(\phi_1 (\tfrac 1 2\cdot A )\big)
- \mathcal M\big (\phi_1 \big (\phi _1 (\tfrac 1 2\cdot A )\big) \big).
\end{align*}
Thus, taking either $A_{\frac 1 2} := \frac 1 2 \cdot A$ or $A_{\frac 1 2} := \phi_1 (\frac 1 2 \cdot A)$ we infer
$$ \frac { \varepsilon } {2^{n+1}}  < \mathcal M \big(A _{\frac 1 2} \big) - \mathcal M \big( \phi_1 (A _{\frac 1 2}) \big).$$
The set $A_{\frac 1 2}$ constructed above is contained in $T^{\frac 1 2} K_{\frac 1 2}$, where $K _{\frac 1 2} =B (K_1,{\frac 1 2})$.

Iterating the above procedure we obtain, for $r_m =\frac 1 {2^m}$, a subset $A_{r_m} \subset T^{r_m} K_m$ where $K_m =B (K_{m-1},{r_m})$ and such that \eqref{eq:m} holds true.

The claim follows since all $K_m$ are contained in the set $B (K_1,1)$, whose closure is compact, by completeness of $X$.
\end{proof}

\begin{rem}
The completeness of the space $X$ is used in the proof of \tref{thmmain} only once, namely in the last line of the above proof.
\end{rem}

\subsection{Relation with the mm-boundary} Let us interpret the deviation measures $\mathcal{V}_r$ from \eqref{eq:first} in suitable geometric terms.

Let $K\subset X$ be measurable and let $r>0$ be arbitrary.
Since $\mathcal H^n (X\setminus X_{reg} )=0$, we have
$$\mathcal M (T^{r} K) =\omega _n \cdot r^n \cdot \mathcal H^n (K).$$
Denote now by $U^r(K)$
the set of all pairs $(x,y)\in X\times X$ with $x\in K$ and $d(x,y)<r$.
By Fubini's theorem the set $U^r(K)$ is $\mathcal H^n \otimes \mathcal H^n =\mathcal H^{2\cdot n}$ measurable and we have
$$\mathcal H^{2\cdot n} (U^r (K))= \int _K b_r (x) \cdot d\mathcal H^n (x). $$
Taking both equations together, we see that the signed measure $\mathcal{V}_r$ expresses the difference between $\mathcal H^{2\cdot n}$ and $\mathcal M$. More precisely,
\begin{equation} \label{eq:compare}
\mathcal{V}_r (K) = \frac 1 {\omega_n \cdot r^n}\cdot \Big(\mathcal M (T^r K)- \mathcal H^{2\cdot n} (U^r (K)) \Big).
\end{equation}

The following statement is a reformulation of \tref{thmmain}.

\begin{thm} \label{reform}
If an Alexandrov space $X$ has vanishing  mm-boundary then it has the Liouville property.
\end{thm}

\begin{proof}
Arguing by contradiction, assume that $X$ does not have the Liouville property.
Consider the compact subset $K \subset X$, the positive numbers $\varepsilon, r_m$ and the Borel
subsets $A_m\subset T^{r_m} K$ provided by \lref{infini}.

Let $Y$ be the closure of $B (K,1)$.
Recall that $D\subset TX$ is the set of all vectors at which the exponential map is defined.
For $r>0$, denote by $D^r$ the intersection of $D$ with $T^r Y$ and consider the ``total exponential map''
$E\:D^r \to X\times X$ given by
$$E(v)= (\pi (v), \pi (\exp (v))).$$
As above, let $U^r =U^r(Y)$ be the set of all
pairs $(y,x) \in X\times X$ with $y\in Y$ and $d(x,y)<r$.
Note that
\begin{equation} \label{eq:image}
E(D^r) =U^r.
\end{equation}
Moreover, for any fixed $x \in Y$, the restriction of $E$ to $D_x \cap D^r$ is
a $(1+ C r^2)$-Lipschitz continuous map from $D_x\subset T_xX$ onto the set $U^r \cap (\{x \} \times X)$ (see Subsection \ref{subsec:Alex}).
Thus, for all sufficiently small $r$, and any Borel subset $S\subset D_x \cap D^r$, we have
$$\mathcal H^n (E(S)) \leq (1+4{\cdot} n{\cdot}  C{\cdot}  r^2)\cdot \mathcal H^n (S).$$
Using the definition of the Liouville measure $\mathcal M$ and Fubini's formula for the product measure
$\mathcal H^{2\cdot n} =\mathcal H^n \otimes \mathcal H^n $ on $X\times X$ we obtain for any $\mathcal M$-measurable subset $S$ of $D^r$
\begin{equation} \label{eq:contract}
\mathcal H^{2\cdot n} (E(S)) \leq (1+4{\cdot} n{\cdot} C{\cdot}  r^2) \cdot \mathcal M(S).
\end{equation}
Due to \eqref{eq:compare}, the vanishing  of the mm-boundary of $X$ implies
$$\lim _{r\to 0} \frac 1 {r^{n+1} }\cdot |\mathcal M(T^{r} Y) -\mathcal H^{2\cdot n} (U^r)| =0.$$
Thus, up to terms of  order higher than $r^{n+1}$, the map $E$ does not increase the measure of subsets, but the total mass of the image
coincide with the total mass of the target.
Therefore, $E$ is measure preserving  up to terms of  order higher than $r^{n+1}$ on all subsets of $T^r Y$.
% \footnote{Is this explanation helpful? \color{red} too long IMO. I shortened it a bit, Vitya}
More precisely, combining the last two formulas we obtain for every $\delta >0$ the existence of some $s>0$ with the following property. For all $0<r<s$ and all measurable subsets $S\subset D^r$
we have
\begin{equation} \label{eq:almostall}
\begin{aligned}
 \mathcal M (T^r Y)- \mathcal M (D^r) &< \delta{\cdot}r^{n+1},
\\
|\mathcal H^{2\cdot n} (E(S)) - \mathcal M (S)| &< \delta{\cdot}r^{n+1}.
\end{aligned}
\end{equation}
For any measurable subset $S\subset D^r \cap TK$ we now claim
\begin{equation} \label{eq:finally}
|\mathcal H^{2\cdot n} (E(S)) - \mathcal M (\phi _1 (S))| < 2{\cdot}\delta{\cdot}r^{n+1}.
\end{equation}
In order to prove \eqref{eq:finally}, let $S^+$ be the subset of all vectors $v\in S $ for which $\phi_1 (v)$ exists and is contained in $TX_{reg}$.
For all $v\in S^+$, we have $-\phi _1 (v) \in D^r$ and, due to \eqref{eq:symm},
$$E(-\phi_1 (v)) =J(E(v)).$$
The involution $I(v)=-v$ is $\mathcal M$-preserving on $TX_{reg}$. And the involution
$J\:X\times X\to X\times X$ given by $J(x,y)=(y,x)$ preserves $\mathcal H^{2\cdot n}$. Therefore, from \eqref{eq:almostall} we deduce
\begin{equation} \label{eq:final+}
|\mathcal H^{2\cdot n} (E(S^+)) - \mathcal M (\phi _1 (S^+))| < \delta{\cdot}r^{n+1}.
\end{equation}
On the other hand, by construction,
$$ \mathcal M (S\setminus S^+ )=0
\quad\text{and}\quad
\phi _1 (S\setminus S^+)\cap TX_{reg} =\emptyset.$$
Hence, applying \eqref{eq:almostall}, we see
\[
|\mathcal H^{2\cdot n} (E(S))- \mathcal H^{2\cdot n} (E(S^+))|< \delta{\cdot}r^{n+1}\]
and
\[\mathcal M (\phi _1 (S\setminus S^+)) =0.\]
Together with \eqref{eq:final+} this finishes the proof of \eqref{eq:finally}.

Coming back to our subsets $A_m\subset T^{r_m} K$, we have
\begin{align*}
\varepsilon \cdot r_m^ {n+1} &\leq \mathcal M ( A_m) - \mathcal M( \phi _1 (A_m)) \leq
\\
&\leq\mathcal M ( A_m) - \mathcal M( \phi _1 (A_m\cap D^{r_m})).
\end{align*}
Setting $S_m=A_m\cap D^{r_m}$ we estimate the right hand side as the sum of the following three terms:
\begin{align*}
&|\mathcal M (A_m) - \mathcal M(S_m)|,
\\
&|\mathcal M (S_m ) - \mathcal H^{2\cdot n} (E(S_m) )|,
\\
&| \mathcal H ^{2\cdot n} (E(S_m)) - \mathcal M (\phi _1(S_m)|.
\end{align*}
Applying \eqref{eq:almostall} and \eqref{eq:finally} this sum is bounded above by $4{\cdot}\delta{\cdot} r _m^{n+1}$, for all large $m$.

Therefore
\[\varepsilon \cdot r_m^ {n+1}< 4{\cdot} \delta{\cdot} r _m^{n+1}\]
for all large $m$.
Since $\delta$ is an arbitrary positive number, this leads to a contradiction.
\end{proof}

\subsection{Quasi-geodesics flow} \label{subsec:quasi}
Finally, we discuss some relations with quasi-geodesics, referring the reader to \cite{Petsemi} for the basic properties of such curves.
Recall, that whenever a unit speed minimizing geodesic $\gamma _v \:[0,a] \to X$ start at a point $x$ in the direction $v$ then this is the unique quasi-geodesic defined on the interval $[0,a]$, \cite{PP}, p.8, thus the same statement is also true for (local) geodesics $\gamma _v$.

Using this and the fact that a limit of quasi-geodesics is a quasi-geodesic, it is not difficult to conclude that the partial geodesic flow $\phi\:\mathcal F\cap TX_{reg} \to TX_{reg}$ defined above is continuous.
The latter statement slightly strengthening Lemma \ref{lem:measurability}.

As in Subsection \ref{subsec:tb}, we have a canonical measure $\mathcal M _1$ on the unit tangent bundle $\Sigma X \subset TX $ of $X$, which we also call the Liouville measure. Whenever
$X$ has the Liouville property, then the geodesic flow is defined $\mathcal M_1\otimes \mathcal H^1$-almost everywhere on $\Sigma X \times \R$ and preserves $\mathcal M_1$. In this case for $\mathcal M_1$-almost each unit direction there exists exactly one quasi-geodesic starting in this direction.

Let now $X$ be an Alexandrov space with topological boundary $\partial X$ and let $Z$ be the doubling $X\sqcup_{\partial X} X$, which
is an Alexandrov space without boundary, \cite{P2}.
Quasi-geodesics in $X$ are exactly the projections of the quasi-geodesics in $Z$ under
the folding $f\:Z\to X$.
From this we deduce that if $Z$ has the Liouville property, then $\mathcal M_1$-almost each direction $v\in \Sigma X$ is the starting direction of a unique infinite quasi-geodesic in $X$.
Moreover, in this case, the corresponding quasi-geodesic flow preserves $\mathcal M_1$.

Finally, as an application of \tref{thmmain} and \tref{alexandrovthm} we see that the above assumptions are fulfilled whenever the complement $X\setminus\partial X$ has vanishing  mm-boundary. Indeed, in this case the mm-boundary of $Z$ must be concentrated on $\partial X \subset Z$, hence it must be trivial by \tref{alexandrovthm},(3).

\section{Surfaces with bounded integral curvature in the sense of Alexandrov} \label{sec:surface}

\subsection{Preparations}
We assume that the reader is familiar with the theory of surfaces with bounded integral curvature; see \cite{AZ} and \cite{Reshetnyak-GeomIV}.

Let $X$ be a surface with bounded integral curvature;
it is a locally geodesic metric space, homeomorphic to a two-dimensional surface.
It has Hausdorff dimension $2$ and the Hausdorff measure $\mathcal H^2$ is a Radon measure on $X$. There is another signed Radon measure on $X$, the so called \emph{curvature measure} which will be denoted $\Omega$, \cite[Section 8]{Reshetnyak-GeomIV}.
We will not assume that $X$ is complete.

We will derive \tref{intsurface} as a consequence of the following weak local
version of a theorem of Mario Bonk and Urs Lang, \cite{Bonk-Lang}, which relate the curvature measure to the volume of balls.

\begin{lem} \label{lem:bl}
There exists some $\delta _0>0$ with the following property.

 Let $X$ be a surface with bounded integral curvature and let $\Omega\in \mathrm M(X)$ be its curvature measure.
Assume $X$ is homeomorphic to a plane and $| \Omega|(X) < \delta _0$.
Then for any point $x\in X$, and $r>0$ such that
$\bar B (x,{r})$ is compact we have
$$\left|1- \frac {b_r(x)} {\pi{\cdot}r^2} \right| \leq 3 \cdot |\Omega |( B (x,{r})).$$
\end{lem}

\begin{proof}

Set $\delta = |\Omega |( B (x,{r}))$. By continuity, it is sufficient to prove that $|1- \frac {b_s(x)} {\pi{\cdot}s^2}| \leq 3{\cdot}\delta $ for any $s<r$. Using approximations of
the metric on $X$ by polyhedral metrics \cite[Theorem 8.4.3, Theorem 8.1.9]{Reshetnyak-GeomIV}, we assume from now on that $X$ is polyhedral and homeomorphic to $\R^2$.

\emph{Claim:} There exists a complete polyhedral surface $\hat X$ homeomorphic to a plane, which contains a copy of $B (x,s)$ and such that
the curvature measure $\hat \Omega$ of $\hat X$ satisfies $|\hat \Omega | (\hat X) < 3 \cdot \delta$.

Once the claim is proven,
\cite{Bonk-Lang} provides us a bi-Lipschitz map $f\:\hat X\to \R^2$ with the constant $L\leq 1+\frac {3{\cdot}\delta} {2{\cdot}\pi -3{\cdot}\delta}$.
Since $\delta$ is small,  an application of \eqref{eq:bilip}  finishes the proof of the lemma.

It remains to prove the Claim,  certainly well-known to experts.
Take some $r>t>s$ and consider the compact metric ball $\bar B (x,t) \subset B (x,r)$.
We may assume that the boundary $S_t$ of
$\bar B (x,t)$ does not contain singular points of $X$.
By \cite[Theorem 9.1, Theorem 9.3] {Reshetnyak-GeomIV}, the boundary $S_t$ is a (piecewise smooth) Jordan curve, once $\delta _0< 2{\cdot}\pi$, and the negative part $\kappa ^-$ of the geodesic curvature $\kappa$ of $S_t$ satisfies $|\kappa ^-| (S_t) \leq \delta$.
Since $X$ is homeomorphic to a plane this implies that $\bar B (x,t)$ is homeomorphic to a closed disk $\bar D^2$ in $\R^2$.

We find a polygonal Jordan curve $\Gamma$ in $B(x,t)$ approximating $S_t$ such that the negative part of the geodesic curvature of $\Gamma$ is smaller than $2{\cdot}\delta$. Consider the closed Jordan domain $Y$ bounded by $\Gamma$, which can be assumed to contain $B(x,s)$.
Now we glue to $Y$ along any edge of $\Gamma$ a flat half-strip.
The boundary of the arising polyhedral surface consists of pairs of rays
$\gamma _i ^{\pm}$ emanating from the vertices $V_1,\dots,V_k$ of $\Gamma$. The rays $\gamma _i^{\pm}$ enclose an angle equal to $2{\cdot}\pi- \alpha _i$, where $\pi-\alpha _i$ is the angle of $\Gamma$ at $V_i$ measured in $Y$. In order to finish the construction of $\hat X$ we glue a flat sector of
angle $\alpha _i$ between $\gamma _i ^{\pm}$, if $\alpha _i >0$ and we glue $\gamma _i^{\pm}$ together if $\alpha _i \leq 0$.
Since $Y$ was a polyhedral disc, the arising space $\hat X$ is a complete polyhedral plane.
All of the singularities of $\hat X$ are contained in $B(x,s) \cup \{ V_1,\dots, V_k\}$. Moreover, by construction, the curvature measure
$\hat \Omega$ of $\hat X$ satisfies
$$\hat \Omega (V_i) = \min \{ 0, \alpha _i \}.$$
We deduce,
%\begin{align*}
$$|\hat \Omega | ( \hat X) = |\Omega (B (x,s))| + |\hat \Omega | ( \Gamma ) \leq \delta + |\kappa ^-| (\Gamma ) < 3{\cdot}\delta.$$

%%\end{align*}

This finishes the proof of the claim and of \lref{lem:bl}.
\end{proof}

\subsection{Local finiteness of mm-curvature}
Now we are ready to prove and prove the following generalization of \tref{intsurface}.

\begin{thm}\label{intsurface1}
Let $X$ be an Alexandrov surface with integral curvature bounds. Then, equipped with the Hausdorff measure $\mathcal H^2$, the space $X$ has locally finite mm-curvature.
\end{thm}
\begin{proof}
Let again $\Omega$ denote the curvature measure of $X$. Let $\delta _0>0$ be sufficiently small and satisfy the conclusion of \lref{lem:bl}.
The statement of \tref{intsurface1} is local, so we need to prove it only in a small neighborhood of any point.
Thus we may (and will) assume that there is a point $x_0\in X$
such that $| \Omega| (X \setminus \{ x_0 \}) < \delta _0$ and that $X$ is homeomorphic to a plane.

Let $A\subset X$ be  compact. Choose some $\varepsilon >0$ such that the closure of $B (A,{2{\cdot}\varepsilon})$ in $X$ is compact and such that,
for any $0<2{\cdot}r<\varepsilon$ the inequality $\mathcal H^2 (B(x_0,{3{\cdot}r})) < \tfrac 1 {\varepsilon}\cdot r^2$ holds true;  see \cite[Lemma 8.1.1]{Reshetnyak-GeomIV}.

Let $r<\varepsilon$ be arbitrary.
For any $x\in B (x_0,{2{\cdot}r})$ we have
%\begin{align*}
$$  b_r (x) =\mathcal H^2 (B (x,r)) \leq \mathcal H^2 (B (x_0,{3{\cdot}r})) \leq\tfrac 1 {\varepsilon}\cdot r^2. $$
%\end{align*}

For any $x\notin B (x_0,{r})$ we have $| \Omega |(B(x,r )) < \delta _0$. Thus, by Lemma~\ref{lem:bl}, $$|1-\frac {b_r(x)} {\pi{\cdot}r^2} |
\leq
3\cdot |\Omega | ( B (x,{r})).$$
For the deviation measures $\mathcal{V}_r$ from \eqref{eq:first} we estimate:
\begin{align*}
|\mathcal{V}_r| (A\cap B (x_0,{2{\cdot}r}))
&\leq |\mathcal{V}_r| (B (x_0,{2{\cdot}r})) \leq
\\
&\leq (1+ \tfrac 1 {\varepsilon}) \cdot \mathcal H^2 (B (x_0,{2{\cdot}r})) \leq
\\
&\leq
(1+ \tfrac 1 {\varepsilon}) \cdot \tfrac 1 {\varepsilon} \cdot r^2.
\end{align*}
On the other hand,
\begin{align*}
|\mathcal{V}_r| (A\setminus B (x_0,{2{\cdot}r}))
&\leq
\int _{A \setminus B (x_0,{2{\cdot}r}) } 3 \cdot |\Omega| (B (x,r)) \cdot d\mathcal H^2 (x)
\leq
\\
&\leq 3\cdot \int _{B (A,r)\setminus B (x_0,{2{\cdot}r}))} \mathcal H^2 (B(x,r)) \cdot d|\Omega | (x),
\end{align*}
where we have used \lref{lem:exchange} in the last step.
For any $x$ contained in the domain of integration of the last integral, we have $\mathcal H^2 (B(x,r)) =b_r (x) \leq 2{\cdot}\pi{\cdot}r^2$, by \lref{lem:bl}, once $\delta _0$ has been chosen to be sufficiently small.
We deduce $|\mathcal{V}_r| (A\setminus B (x_0,{2{\cdot}r})) \leq 6 {\cdot}\pi {\cdot}\delta _0 {\cdot} r^2$.

Thus, for some constant $C=C(\varepsilon)$ and all $r<\varepsilon$, we obtain
$$|\mathcal{V}_r|(A)  = |\mathcal{V}_r|  (A\setminus B (x_0,{2{\cdot}r})) + |\mathcal{V}_r| (A\cap B (x_0,{2{\cdot}r}))    \leq C\cdot r^2.$$
This finishes the proof of the theorem.
\end{proof}

\section{Convex hypersurface} \label{sec:hyper}
In this section we are going to prove \tref{thmfirst}.

The proof will follow from \tref{thmmain} by comparing the mm-boundary with the mean curvature measure on convex hypersurfaces.

It is possible to deduce the theorem without a reference to \tref{alexandrovthm}, from \lref{lem:mean} alone, but \tref{alexandrovthm} shortens the proof.

All results in this section are local,
but for simplicity, we consider only closed convex hypersurfaces.
The hypersurfaces will always be equipped with the induced intrinsic metric.

We assume that the reader is familiar with the basics of the theory of convex functions and convex geometry.

\subsection{Mean curvature}
Let $X$ be a convex hypersurface in $\R^{n+1}$. Recall that there exists a Radon measure $\mathcal K$ on $X$, called the mean curvature measure; see \cite{Schneider}, \cite{Fedcurvature}.

The measure $\mathcal K$ has the following properties.
For smooth hypersurfaces $X$, we have $\mathcal K=\kappa \cdot \mathcal H^n$, where $\kappa$ is the usual mean curvature function of $X$.
The mean curvature measure is stable under Hausdorff convergence of convex hypersurfaces in $\R^{n+1}$.
If the hypersurface is rescaled by $\lambda$, the mean curvature $\mathcal K$ is rescaled by $\lambda^{n-1}$. %???

A point $x$ in the convex hypersurface $X$ is called \emph{smooth} if there is a unique supporting hyperplane of $X$ at this point.
For any smooth point $x\in X$, any sequence $x_j\in X$ converging to $x$ and any sequence of positive numbers $t_j$ converging to $0$, the
sequence of convex hypersurfaces $X_j$ obtained from $X$ by the dilatation by the factor $\frac 1 {t_j}$ centered at the point $x_j$ converges
to the tangent hyperplane of $X$ at $x$.

The stability of the mean curvature measures $\mathcal K$, vanishing of $\mathcal K$ on flat hyperplanes and
the behavior of $\mathcal K$ under rescalings gives us:

\begin{lem} \label{lem:compsm}
Let $X$ be a convex hypersurface in $\R^{n+1}$. Let $A$ be a compact set of smooth points in $X$ and  $\delta >0$.
Then there exists some $t>0$ such that
\[\mathcal K(B(y,r)) \leq \delta \cdot r^{n-1}\]
for any $y\in B (A,{t})$ and any $0<r<t$.
\end{lem}
Thus, the following lemma applies to all small balls in a neighborhood of any smooth point.

%The following lemma is basic:

\begin{lem} \label{lem:mean}
There exist numbers $\delta_0,C>0$ depending only on $n$ with the following property. Let $X$ be a convex hypersurface in $\R^{n+1}$.
Let $x\in X$ be a point and $r>0$ be such that the mean curvature $\mathcal K$ satisfies $\mathcal K(B (x,{6{\cdot}r})) < \delta \cdot r^{n-1}$ with $\delta <\delta _0$.
Then
\begin{equation} \label{eq:lemma}
\left| 1 - \frac {b_r (x)} { \omega _n{\cdot}r^n}\right|
<
C\cdot \delta \cdot \mathcal K( B (x,{6{\cdot}r})) \cdot r^{1-n}.
\end{equation}
\end{lem}

\begin{proof}
By rescaling, it suffices to prove the existence of $\delta _0, C>0$ such that the lemma holds for $r=1$.
By approximation, it is sufficient to prove the result for smooth convex hypersurfaces.

Fix a sufficiently small $\varepsilon _0>0$.
The mean curvature vanishes on $B (x,6)$ if an only if $B(x,6)$ is contained in a flat hyperplane.
Due to the stability of $\mathcal K$ under convergence, if $\delta_0$ is small, then
the ball $U=B (x,{5}) \subset X$ is close to a flat hyperplane in $\R^{n+1}$.
Thus, we may assume that the tangent hyperplanes to points in $U$ are $\varepsilon_0$-close to the tangent
space $W =T_xX \subset \R^{n+1}$. Therefore, $U$ is a graph $U=\{(x,f(x))\}$ of a convex function $f\:V\to \R$ defined on an open subset $V\subset W$. Moreover, $V$ contains the ball of radius $4$ in $W$ around $x$.
Denote by $B(x,2)_W$ the ball of radius $2$ in $W$ around $x$.
Set
$$a :=\sup \set{|\nabla f (y)|}{y\in B(x,2)_W}.$$
If $\delta_0$ is small, then $a<\varepsilon_0$.
The orthogonal projection $P\:U\to V$ is $1$-Lipschitz and the restriction of the inverse $P^{-1}$ to $B(x,2)_W$ has Lipschitz constant $$\sqrt {1+ a^2} \leq 1+a^2 \leq 1+\varepsilon_0 ^2.$$
Applying \eqref{eq:bilip} we only need to prove that $a < C \cdot \delta $ for a constant $C$.

Denote by $|D^2f|$ the largest eigenvalue of the Hessian $D^2f$.
Since $f$ is convex and $\varepsilon _0$ is small, the mean curvature $\mathcal \kappa (x)$ at the point $(x,f(x))$ of the graph $U$ of $f$ satisfies $\kappa (x) \geq \frac n 2 \cdot |D^2 f|$. Hence, the conclusion follows from the following statement.

\emph{Claim:} Let $f:B \to \R$ be a smooth convex function on the open ball $B=B (0,4)\subset \R^n$.
If $f(0)=|\nabla f(0)| =0$ then, for some $C=C(n)>0$,
$$\sup_{y\in B (0,2)} |\nabla f (y)| \leq C\cdot \int _{B} |D^2f|.$$

By convexity, it is sufficient to find some $C=C(n)>0$ with
\begin{equation} \label{eq:vitya}
\sup _{y\in B (0,3) }| f(y)| \leq C \cdot \int _{B} |D^2f|;
\end{equation}
see also \cite[Theorem 6.7]{Evans}.

First note that $f(z)\ge 0$ for all $z$ since  $f(0)=|\nabla f(0)| =0$ and $f$ is convex.

In order to verify \eqref{eq:vitya}, we can multiply the function $f$ by a constant and assume that
$f$ takes its maximum on the closed ball $\bar B (0,3)$ at the point $y_0$ and $f(y_0)=1$.
Convexity of $f$ implies that $|y_0|=3$.
Since $f(0)=0$ and $f$ is convex, we must have $f(y)\leq \frac 1 3$ for all $y\in B(0,1)$.

By convexity and the choice of $y_0$, the restriction of $f$ to the supporting hyperplane $H$ of $\bar B(0,3)$ at $y_0$ is bounded from below by $1$. Consider the ball $S$ of radius $\frac 1 2$ in $H$ around $y_0$. For any point $z\in S$ consider the restriction
$$f_z (t)= f(z- \tfrac t 3\cdot y_0 ),\quad t\in [0,6]$$
to the segment of length $6$ starting at $z$ orthogonal to $H$.
Then
$$f_z(0)\geq 1,\quad
f_z (3)\leq \frac 1 3,\quad
f_z (6) \geq 0.$$
Thus for some $t\in (0,3)$ we have $f_z'(t) \leq -\frac 2 9$ and for some $t\in (3,6)$ we have $f_z'(t) \geq - \frac 1 9$.
Therefore
$$\int _0 ^6 f_z'' (t)\cdot dt \geq \frac 1 9.$$
Integrating over $S$ we obtain by Fubini's theorem a uniform positive lower bound on $\int _{B} |D^2f|$.
This finishes the proof of \eqref{eq:vitya}.
Hence the claim and Lemma follow.
\end{proof}

\subsection{The proof} The next theorem is the first part of \tref{hypersurface};
the second part follows from \tref{intsurface1}.
In combination with \tref{thmmain} it also finishes the proof of \tref{thmfirst}. %???

\begin{thm} \label{thmconv}
Let $X$ be a convex hypersurface in $\R^{n+1}$.
Then it has vanishing  mm-boundary.

\end{thm}

\begin{proof}
Since $X$ has locally finite mm-boundary by
\tref{alexandrovthm}, it suffices to prove that any partial limit measure $\nu$ of a sequence $\frac 1 {r_j}\cdot \mathcal{V}_{r_j}$ for $r_j\to 0$ must be the zero measure.

Fix a partial limit measure $\nu$.
Due to \tref{alexandrovthm}, $\nu (A)=0$ for any Borel subset $A\subset X$ with $\mathcal H^{n-1} (A)<\infty$.
Let $Y\subset X$ be the set of smooth points of $X$.
The complement $X\setminus Y$ is a countable union of subsets with finite $(n-1)$-dimensional Hausdorff measure (see \cite{zajicek} and \cite[Theorem 1.4]{Schneider}) therefore $\nu (X\setminus Y) =0$. Therefore, it is sufficient to prove $\nu (A)=0$
for any compact subset $A\subset Y$.

Fix a compact subset $A\subset Y$ and
let  $\delta>0$ be an arbitrary sufficiently small number. Consider a positive $1>t>0$ provided by \lref{lem:compsm}. Let $U$ be the open set  $B (A,t)$.
%Making $t$ smaller, if needed, we may assume that $\mathcal K(B(A,{7{\cdot}t}))$ is bounded by a
%fixed constant $C_1$ independent of $\delta$.

Assume $0<r<t$.
Applying \lref{lem:mean}, for $x\in U$ we get
\begin{align*}
|\mathcal{V}_r| (U)
&\leq \int _U C\cdot \delta \cdot r^{1-n}\cdot \mathcal K (B (y,{6{\cdot}r})) \cdot d\mathcal H^n (y) \leq
\\
&\leq C\cdot \delta \cdot r^{1-n} \cdot \int _{B (A,{7{\cdot}t})} \mathcal H^n (B (y,{6{\cdot}r})) \cdot d\mathcal K(y)
\leq
\\
&\leq C\cdot \delta \cdot r^{1-n} \cdot (6{\cdot}r) ^n \cdot \mathcal K(B (A,{7{\cdot}t}));
\end{align*}
we have used \lref{lem:exchange} in the second and Bishop--Gromov inequality in the last inequality. Hence
$$|\nu| (A) \leq |\nu| (U) \leq C\cdot \delta \cdot 6^n \cdot \mathcal K(B (A,7)). $$
Since $\delta$ can be chosen arbitrary small, we obtain $|\nu| (A)=0$.

This finishes the proof of the claim and, therefore, of \tref{thmconv}.
\end{proof}

%\subsection{Generalization}
%The following more general statement seems to follow in the same way. Maybe it should be better steted after the section about
%BV-metrics.

%\begin{defn}
%A $DC$-submanifold $N$ of a smooth submanifold $M$ is a subset that can be locally around each point represented as a graph
%$N=(x,f(x))$ for a $DC$-map $f\:\R ^m \to \R ^n$.
%\end{defn}

%Main examples are $C^{1,1}$ submanifolds and convex hypersurfaces in Riemannian manifolds.
%Using the theorem of Nash, one can (isometrically) consider any $DC$-submanifold of a Riemannian manifold
%as a $DC$-submanifold of a euclidean space.

% It seems to me that everything explained above workes well in this more general situation and one obtains:

% \begin{prop}
% Let $N$ be a $DC$-submanifold of a smooth Riemannian manifold. Then $N$ has trivial mm-boundary.
% \end{prop}

\section{An integral inequality for Riemannian metrics}\label{sec-BV-estimate}
\subsection{The smooth case} We start by estimating from above the deviation measure $\mathcal{V}_r$ on a smooth Riemannian manifold in terms of the first
derivatives of the metric. We do not know how to prove a similar estimate from below, see Problem \ref{qe:BV}. However, for the applications
to Alexandrov spaces discussed in the next section, the estimate from below is a consequence of the theorem of Bishop--Gromov.

For a smooth Riemannian metric $g$ defined on an open subset $U\subset \R^n$ we denote by $|g'|\:U\to [0,\infty)$ the sum $\Sigma_{i,j,k}|\frac \partial {\partial \,x_k} g_{ij}|$.

%of the absolute values of the first partial derivatives $|\frac \partial {\partial \,x_k} g_{ij}|$ of the coordinates of $g$.

\begin{prop}\label{prop-smooth}
There exists a constant $C=C(n)>1$ with the following property.
Let $U\subset \R^n$ be an open subset with a smooth Riemannian
metric $g$ which is $(1+\frac 1 C)$-bi-Lipschitz to the background Euclidean metric.
%t $g_0$ denote the flat background metric and assume $\|g_x-g_0\| \leq \varepsilon$ for some fixed small
%$\varepsilon$.
Let $A\subset U$ be a Borel subset. Let $r>0$ be such that $B (A,{2{\cdot}r})$ is relatively compact in $U$.
Then
$$\mathcal{V}_r (A) \leq C \cdot r \cdot \int _{B ( A,{2{\cdot}r})} |g' |.$$
\end{prop}

\begin{proof}
We will denote by $C_i$ various (explicit) constants which depend only on $n$.

We will use the following notations. By $|\cdot|$ and $\mathcal L^n$ we denote respectively the norm and the Lebesgue measure on the background $\R^n$.
For $x\in U$ we denote by $g_x$ the Riemannian tensor at the point $x$ and
by $|\cdot |_x$ the corresponding norm. % $\R^n$ will be considered with the Euclidean metric coming from $g_0$
% and the corresponding Lebesgue measure $\mathcal L^n$.
The Hausdorff measure of the Riemannian metric
$g$ has the form $u\cdot \mathcal L^n$, with $u=\sqrt {det (g_{ij})}$.

For $x\in \R^n$, we consider the function $K \:U\to [0,\infty)$ given by
$$K (x)=\sup _{|v|_x =1}\, \frac d {dt}\Big |_{t=0} |v| _{x+tv}.$$
By smoothness of the determinant and the square root,
we find a constant $C_1$ such that for all $x\in U$ we have
\begin{equation} \label{eq:Kg}
|u' (x)| \leq C_1\cdot |g' (x)|
\quad\text{and}\quad
K (x) \leq C_1 \cdot |g' (x)|.
\end{equation}
We fix $A\subset X$ and $r>0$ as in the formulation of the proposition.
For $x\in U$ denote by $B_x$
the metric ball $B(x,r)$ in $U$.
By $B^x$ we denote the metric ball of radius $r$ in the Euclidean norm $|\cdot |_x$. In this Euclidean metric the ball $B^x$ has measure
$$\omega _n\cdot r^n =u(x)\cdot \int _{B^x} d\mathcal L^n.$$
Thus, in order to estimate the deviation measure $\mathcal{V}_r$, we only need to control the summands on the right of the following inequality:
\begin{equation} \label{eq:summand}
\omega_n \cdot r^n - b_r (x) \leq u(x) \cdot \mathcal L^n( B^x\setminus B_x) + \int _{B_x} |u(x)- u(y)| \cdot
d\mathcal L^n (y).
\end{equation}

We may assume that the bi-Lipschitz constant $1+\frac 1 C$ is close to 1, so that $\frac 1 2 < u <2$. Moreover, we may assume $B_x$ and $B^x$ are contained in the ball of radius $\frac 4 3 r$ around $x$ with respect to the background Euclidean metric.

In order to bound the first summand, for $x\in A$ and $|v|_x=1$, we set $l^x_v$ to be the length of the segment $[x,x+v]$
in the Riemannian metric~$g$. Then we compute
\begin{align*}
l_v ^x -r
&=
\int _0 ^r |v|_{x+tv} \, dt - \int _0 ^r |v|_x \, dt
\leq
\\
&\leq \int _0 ^r (\int _0 ^t K(x+sv) \, ds)\, dt
\leq
\\
&\leq \int _0 ^r (\int _0 ^r K(x+sv) \, ds)\, dt
=
\\
&= r\cdot \int _0 ^r K(x+sv) \, ds.
\end{align*}
Observe now that the intersection of $B^x\setminus B_x$ with the ray starting in $x$ in the direction of $v$ has $\mathcal H^1$-measure
( with respect to the norm $|\cdot |_x$)
at most $2\cdot (l_v^x -r)$, once the bi-Lipschitz constant $(1+\frac 1 C)$ is close to $1$.
Integrating in polar coordinates over the ball $B^x\subset (\R^n, |\cdot|_x)$ we infer:
\begin{align*}
u(x)\cdot \mathcal L^n (B^x\setminus B_x)
&\leq r^{n-1} \cdot \int _{v\in S^{n-1} _x}
\Big( 2\cdot r \cdot \int _0 ^r K(x+sv) \, ds \Big)
\cdot d\mathcal H^{n-1}  =
\\&=2{\cdot}r^n \cdot \int _{B^x} K(y) \cdot |y-x|^{n-1} \cdot u(x) \cdot d\mathcal L^n,
\end{align*}
where $S^{n-1} _x$ is the unit sphere in $(\R^n, |\cdot|_x)$.

To get a similar estimate of the other summand in \eqref{eq:summand},
we only need to recall the following inequality from \cite[Lemma 4.1]{Evans}, valid for any $\mathcal C^1$ function $u$ on a Euclidean ball
\[
\int _{|x-y| < r} |u(y)-u(x)| \cdot d\mathcal L^n (y) \leq C_2\cdot r^{n}\cdot \int _{|x-y| < r}|u'(y)| \cdot |y-x| ^{1-n} \cdot d\mathcal L^n (y).
\]
Taking both estimates together with \eqref{eq:Kg}, embedding $B_x$ and $B^x$ in slightly larger Euclidean balls and using that $\frac 1 2 < u <2$, we conclude:

\[
\omega _n \cdot r^n - b_r (x) \leq C_3\cdot r^n \cdot \int _{|x-y|< \frac 4 3 r} |g'(y)| \cdot |y-x|^{1-n} \cdot d\mathcal L^n.
\]

We divide both sides by $\omega _n \cdot r^n$ and integrate over $A$. Using that the bi-Lipschitz constant is close to $1$, we see:
\begin{align*}
\mathcal{V}_r (A) &\leq C_4\cdot \int _A \Big (\int _{|x-y|< \frac 4 3 r} |g'(y)| \cdot |y-x|^{1-n}\cdot d\mathcal L^n (y) \Big ) \cdot d\mathcal L^n (x) \leq
\\
&\leq C_4\cdot \int _{B (A,{2{\cdot}r})} \Big ( \int _{|x-y|< \frac 4 3 r} |g'(y)| \cdot |y-x|^{1-n} \cdot d\mathcal L^n (x) \Big ) \cdot d\mathcal L^n (y) =
\\
&= \frac 4 3\cdot  C_4\cdot \int _{B (A,{2{\cdot}r})} |g'(y)| \cdot r \cdot d\mathcal L^n (y),
\end{align*}
where we have used \lref{lem:exchange} in the second inequality.
This finishes the proof of Proposition~\ref{prop-smooth}.
\end{proof}

%\begin{rem}
%Estimating $\mathcal{V}_r (A)$ from below seems to be more difficult. We do not need a lower estimate in the general case in this paper because in our applications $\mathcal{V}_r$ is nonnegative due to Bishop--Gromov volume comparison. \footnote{ {V.:\color{red} do we need to go into this?} A.: No it is not needed. Should we mention the problem somewhere or just delete it? }
%\end{rem}

\subsection{Functions of bounded variations}
%before proceeding let us briefly recall some known facts about $DC$ functions and $DC$ maps.
Let $U$ be an open subset of $\R^n$. A function $f\in L^1 (U)$ is of class BV (bounded variation) if its first partial derivatives,
$\frac{\partial f}{\partial x^i}$ (here and below always in the sense of distributions) are signed Radon measures with finite mass
$|\frac{\partial f}{\partial x^i}|(U)$. We denote by $[Df ]$ the Radon measure $\sum_i |\frac{\partial f}{\partial x^i}|$ on $U$.
If $f\:U\to \R$ is a BV function, which is continuous on a subset $R \subset U$ with $\mathcal H^{n-1} (U\setminus R)=0$ then the Radon measure $[Df]$ vanishes on all Borel subsets $A\subset U$ with $\mathcal H^{n-1} (A)<\infty$, \cite{Goffmann}.

Let $f\:U\to \R$ be of class BV. Then for $\mathcal H^n$-almost every point
$x\in U$ there exists an affine function $\hat f _x\:\R^n\to \R$, such that for the BV function $h_x=f-\hat f_x$ we have
\begin{equation}\label{DC-ae-diff}
\lim_{r\to 0}\frac 1 {r^{n+1}}\cdot \int _{B(x,r)}|h_x| =0
\quad\text{and}\quad
\lim _{r\to 0} \frac 1 {r^n} \cdot [Dh_x] (B (x,r)) =0;
\end{equation}
see \cite[ Theorem 6.1 (2),(3)]{Evans} for the second and the H\"older inequality and \cite[ Theorem 6.1 (1)]{Evans} for the first inequality.
%A function $f\co U\to \R$ is called $DC$ if it can be locally represented as $h-g$ where $h$ and $g$ are semiconcave.
%We refer to \cite{Per-DC} and \cite{Shioya} for more details on the results about DC-functions used below.

%A map $f=(f_1,\ldots, f_m)\co U\to R^m$ is called $DC$ if each of its coordinates is $DC$.
%$DC$ functions form an algebra and a composition of two $DC$ maps is $DC$. If $f,g$ are $DC$ functions and $g$ is never $0$ on $U$ then $\frac f g$ is $DC$.
%Partial derivatives of $DC$ functions, here and below always in the sense of distributions, are functions of bounded variations, abbreviated as $BV$-functions below.
%Any partial derivative of a function of bounded variation is a signed Radon measure. Therefore (distributional) second partial derivatives of $DC$ functions are signed Radon measures.

%Just like semiconcave functions, $DC$ functions have second differential a.e. and their partial derivatives which exist a.e. are differentiable a.e.

\subsection{Almost Riemannian metric spaces}
The following definition provides a suitable description of a large part of any Alexandrov space, see Section~\ref{sec:Alex}.

Let $C =C(n)$ be the constant determined in Proposition \ref{prop-smooth}.
We will call a locally geodesic metric space
$X$ an \emph{almost Riemannian metric space} if it has the following properties
(see \cite{AB15} for a careful discussion
of such $DC_0$-Riemannian manifolds in the language of \cite{AB15} and \cite{Per-DC}):
\begin{enumerate}
%\item $X$ is an $n$-dimensional Lipschitz manifold.
\item There is a Borel subset $R\subset X$, called the subset of \emph{regular points} with $\mathcal H^{n-1} (X\setminus R)=0$.
\item Any minimizing geodesic $\gamma$ in $X$ can be approximated by curves $\gamma _i$ in $R$, such that the lengths of $\gamma _i$ converge to the length of $\gamma$.
\item For any $x\in X$, there is a neighborhood $U$ of $x$, called a \emph{regular chart}, and a bi-Lipschitz map
$\phi\:U\to O$ onto an open subset $O\subset \R^n$, with the bi-Lipschitz constant less than $(1+\frac 1 C)$.
\item There is a continuous Riemannian tensor $g_{ij} $ on $\phi (U\cap R)$ such that $g_{ij}$ is a function of bounded variation on $O$
for each $1\leq i,j \leq n$.
\item The length of any curve $\gamma \subset R$ can be computed as the length of $\phi (\gamma )$ via this Riemannian tensor $g$.
% \item The metric on $X$ can be reconstructed from the restriction to $R$. \footnote{I do not now how to express this in a good way.}.
\end{enumerate}

For any regular chart $U$ as above, we set $\mathcal N_0 (U)$ to be the Radon measure $[g']$ on $U$ given as the sum of
the Radon measures $[D g_{ij}]$ over the coordinates $ g_{ij}$ of the metric tensor $g$.
For an almost Riemannian metric space $X$, we define an outer measure $\mathcal N$ on $X$ in the following way.
For a subset $A\subset X$, we consider all coverings $A\subset \cup U_i$ by countably many regular charts $U_i$ and let $\mathcal N(A)$
to be the infimum of the sums $\sum_i \mathcal N_0 (U_i)$ over all such coverings. This is indeed an outer measure, which takes finite values on compact subsets. Since $\mathcal N$ satisfies the Caratheodory criterion, \cite[Theorem 1.9]{Evans}, it is indeed a Radon measure. We will call $\mathcal N$ the \emph{minimal metric derivative measure} on the almost Riemannian metric space $X$.

%We observe:
\begin{lem} \label{lem:minderiv}
Let $X^n$ be a almost Riemannian metric space and let $\mathcal N$ be its minimal metric derivative measure.
Then $\mathcal N (A)=0$ for any Borel subset $A\subset X$ with $\mathcal H^{n-1} (A)<\infty$.
There exists a Borel subset $C\subset X$ of full $\mathcal H^n$-measure in $X$ with
$\mathcal N(C)=0$, thus $\mathcal N$ is absolutely singular with respect to $\mathcal H^n$.
\end{lem}

\begin{proof}
Clearly, both claims are local. Hence we need to verify them only in a regular chart $U$, which we identify with its image $\phi (U) \subset \R^n$. The first statement
follows directly from the continuity of the metric tensor $g$ on the subset $U\cap R$ and the result of \cite{Goffmann}
cited above.

In order to verify the second claim we only need to show the following statement;
see also \cite[Section 1.6]{Evans}.
For almost all $x\in U$ there is another
regular chart $x\in V$, such that the derivative measure $[h']$ of the Riemannian tensor $h$ in this chart $V$, has $n$-dimensional density $0$
at $x$, thus
\begin{equation} \label{eq:deriv}
\lim _{r\to 0} \frac 1 {r^n} \cdot {[h'] (B (x,r))} =0 .
\end{equation}
Here and below, the ball $B (x,r)$ over which we integrate can be equally considered with respect to the Euclidean or to the original metric on $U$, since both are bi-Lipschitz equivalent.
In order to prove \eqref{eq:deriv}, we follow \cite[Section 4.2]{Per-DC} and consider the Riemannian tensor $g$ of the original chart $U$. Applying ~\eqref{DC-ae-diff} to the coordinates of $g$, we find for $\mathcal H^n$-almost all
$x\in U$ a smooth symmetric $2$-tensor $\hat g =\hat g_x$ on $U$ such that for
$u=g-\hat g$ we have:
\begin{equation} \label{eq:u}
\lim_{r\to 0}\frac 1 {r^{n+1}}\cdot \int _{B(x,r)} \|u\| =0
\quad\text{and}\quad
\lim _{r\to 0} \frac 1 {r^n} \cdot [Du] (B (x,r)) =0  .
\end{equation}
The first statement implies that $\hat g$ is indeed a Riemannian metric in a neighborhood $U_0$ of $x$.

Fix such a point $x$, neighborhood $U_0$ and $\hat g$.
Consider a small neighborhood $W$ of $0$ in $\R^n$ and let $\xi \:W\to U$ be the exponential map with respect to the metric
$\hat g$. Then $\xi (0)=x$, $D\xi (0) =Id$ and the pull-back Riemannian metric $\hat h= \xi^{\ast} (\hat g)$ has zero derivative at $0$. Since $D\xi$ is the identity, the bi-Lipschitz constant of the restriction $F=\xi ^{-1} \circ \phi$ to a sufficiently
small neighborhood $V$ of the point $x$ is still less than $(1 + \frac 1 C)$.
Hence, $F\:V\to \R^n$ is a regular chart.

The Riemannian tensor $h$ in this chart equals $\hat h +\xi^{\ast } (g-\hat g)$.
Now, $D\hat h (0) =0$, thus
\eqref{eq:deriv} holds for $\hat h$ instead of $h$. For the other summand $\xi ^{\ast} (u)$, the density estimate
\eqref{eq:deriv} follows from \eqref{eq:u} and the fact that $\xi$ is a $C^2$-diffeomorphism
if $W$ is sufficiently small. This finishes the proof of \lref{lem:minderiv}.
\end{proof}

%Condition (ii) implies that the restriction of the length structure to $R$ determines the metric on $X$, at least locally.

\subsection{The upper bound on the deviation measures}
$ $
Continuing to denote by $C=C(n)$ the constant from Proposition~\ref{prop-smooth} we show:

\begin{cor} \label{cor-dc-vr}
Let $U$ regular chart of an almost Riemannian metric space $X$. Identifying $U$ with its image $O=\phi (U)$, let
$g$ be the metric tensor and the measure $\mathcal N_0 =[Dg]$ the derivative of the metric tensor. For any Borel subset $A\subset U$ and any $r$
such that $B (A,{3{\cdot}r})$ is relatively compact in $U$ we have

\[
\mathcal{V}_r (A) \leq 2 \cdot C \cdot \mathcal N_0 (B (A,{3{\cdot}r}))\, .
\]
\end{cor}

\begin{proof}
Consider a relatively compact open subset $ V \subset U$, which contains $B (A,{2{\cdot}r})$. Apply (coordinatewise) the standard mollifying construction
to the Riemannian tensor $g$. For all small positive $\varepsilon$, we thus obtain smooth metrics $g_{\varepsilon}$ on $V$ with the following properties.
The total derivatives $|g_{\varepsilon} '|$, considered as measures, satisfy $|g_{\varepsilon} '|\leq \mathcal N_0$ on $V$, \cite[Theorem 5.3.1]{Ziemer}. Since $g$ is pointwise $\frac 1 C$-close to the background Euclidean inner product, the same is true for $g_{\varepsilon}$. For all sufficiently small $\varepsilon$ the $2{\cdot}r$-tubular neighborhood around $A$ with respect to $g_{\varepsilon}$ is contained in the $3{\cdot}r$-tubular neighborhood around $A$ with respect to the original distance in $X$. Moreover, $g_{\varepsilon}$ converges to $g$ pointwise at all points of $R$, \cite[Theorem 1.6.1]{Ziemer}.

Denote by $d_{\varepsilon}$ the distance function induced by $g_{\varepsilon}$. From the last statement and the properties (2),(5) in the definition of an almost Riemannian metric space we deduce that
\[
\lim _{\varepsilon \to 0} \sup \set{|d_{\varepsilon} (x,y) -d(x,y)|}{x,y \in V, d(x,y) <r}=0  .
\]
Finally, the Hausdorff measures of the Riemannian metrics $g_{\varepsilon}$ converge on $V$ to the Hausdorff measure of $V$ with respect to the original metric.

Now the result follows directly from \pref{prop-smooth} applied to the metrics $g_{\varepsilon}$, by letting $\varepsilon$ go to $0$.
\end{proof}

As a consequence of \cref{cor-dc-vr}, the minimal metric derivative measure bounds from above the deviations measure $\mathcal{V}_r$
on any almost Riemannian metric space:

\begin{lem} \label{cor-mu-dc}
Let $X$ be a almost Riemannian metric space with the metric derivative measure $\mathcal N$.
Then for any compact subset $A\subset X$, there exists some $r_0>0$ such that for all
$r<r_0$ we have
$$\mathcal{V}_r (A) \leq r\cdot 2\cdot  (n+2)\cdot C\cdot \mathcal N (A)\,. $$
\end{lem}

\begin{proof}
Cover $A$ by finitely many regular charts $U_i$ such that $\sum \mathcal N_0 (U_i)$ is sufficiently close to $\mathcal N (A)$.
Since the covering dimension of $X$ is $n$, we find a finite covering $V_j$ of $A$, which refines the first covering but has intersection multiplicity less than $(n+2)$. Considering each $V_j$ as a subchart of the corresponding chart $U_i$ we see that
$$\sum \mathcal N_0 (V_j) \leq (n+2)\cdot \sum \mathcal N(A)  .$$
Consider $r_0>0$ such that for any $x\in A$ the ball $B (x,{4{\cdot}r_0})$ is contained in one of the sets $V_j$. Denote by $A_j$ the set of all such $x$.
Then $\mathcal{V}_r (A) \leq \sum \mathcal{V}_r (A_j)$ and, due to \cref{cor-dc-vr}, $\mathcal{V}_r (A_j) \leq 2 \cdot C\cdot N_0 (V_j)$.
Combining these inequalities finishes the proof.
\end{proof}

\section{Alexandrov spaces} \label{sec:Alex}
\subsection{Strained points} \label{subsec:strainer}
Strainers and strainer maps are basic tools for Alexandrov spaces; see also \cite{BGP}, \cite{Otsu-Shioya}, \cite{Shioya} and will play an important role in the proof of \tref{alexandrovthm}.

Let us list main properties of the subsets of strained points.
We fix a natural number $n$.
  Then for all sufficiently large $A$ and any $0<r,\delta \leq \frac 1 {A^2}$
  the following properties hold true for all $n$-dimensional Alexandrov spaces $X$ of curvature $\geq -1$:
\begin{enumerate}

\item The set $X_{r,\delta}$ of points in $X$ which have an $Ar$-long $(n,\delta)$-strainer is open in $X$. For $s<r$, we have $X_{r,\delta} \subset X_{s,\delta}$, \cite[9.7]{BGP}.

\item Assume a sequence $(X_i ^n,x_i)$ of Alexandrov spaces of curvature $\geq -1$ converges to an $n$-dimensional Alexandrov space $(X,x)$ in the pointed Gromov--Hausdorff topology.
If $x\in X_{r,\delta}$ then, for all large $i$, the point $x_i$ has an $Ar$-long $(n, \delta)$-strainer in $X_i$.

\item Rescaling $X$ with a constant $\lambda \geq 1$ sends the subset $X_{r,\delta}$ to a subset of $(\lambda X )_{\lambda r, \delta} $ of the rescaled Alexandrov space $\lambda X$.

\item The union $X_{\delta}:= \cup _{r>0} X_{r,\delta}$ contains the set $X_{reg}$ of all regular points of $X$. The Hausdorff dimension of the set $X\setminus (X_{\delta} \cup \partial X)$ is at most $n-2$~\cite[10.6, 10.6.1, 12.8]{BGP}.

\item For any point $x\in X_{r,\delta}$ there are natural distance coordinates $\phi\: B (x,{3{\cdot}r}) \to \R^n$ which are $(1+\varepsilon)$-bi-Lipschitz onto an open subset $O \subset \R^n$. Here, $\varepsilon\to 0$ as $A\to\infty$, \cite[9.4]{BGP}.

\item The chart $\phi$ can be smoothed to satisfy the following property, \cite[Theorem B]{Otsu-Shioya}.
There exists a continuous Riemannian metric $g$ on $\phi (X_{reg} \cap B (x,{3{\cdot}r})) \subset O$ such that for any curve $\gamma \subset X_{reg} \cap B(x,{3{\cdot}r})$ its length
coincides with the length of $\phi (\gamma )$ with respect to the Riemannian metric $g$.

\item The metric tensor $g$ on a chart $O$ defined above is of bounded variation on $O$, \cite[4.2]{Per-DC} (see also \cite{AB15}).
\end{enumerate}

The last three statements in the above list together with the density and convexity of the set $X_{reg}$ of regular points imply the following.

\begin{cor}
In the above notations, the subset $X_{\delta} \subset X$ is an almost Riemannian metric space, once $A$ is sufficiently large.
\end{cor}

In fact, the arguments in \cite[4.2]{Per-DC}, provide a slightly more precise version of (7) in the above list:

\begin{lem} \label{lem:A}
In the notations above, the constant $A$ can be chosen sufficiently large, so that the following holds true.
The derivative measure $[g']$ of the Riemannian tensor $g$
in the canonical distance chart $O$ satisfies $[g'] (\hat O) \leq A \cdot r^{n-1} $, where $\hat O$ is the image $\phi (B (x,{2{\cdot}r})) \subset O=
\phi (B (x,{3{\cdot}r}))$.
\end{lem}

\begin{proof}
We only sketch the proof, referring to \cite{Per-DC} for details.
First we fix $r=\frac 1 {A^2}$.

The fact that $g$ has bounded variation in the chart $O$ follows in \cite[Section 4.2]{Per-DC}, by writing the coordinates of $g$
as a universal smooth map $\Phi (f_1,\dots,f_{\alpha})$ of a finite number of distance functions $f_j$ on $X$ and their partial derivatives, both expressed in the chart $\phi$.
It is shown in \cite[Section 3]{Per-DC}, that any such distance function $f_j$ is expressed in the chart $O$ as a difference of two $L$-Lipschitz and $\lambda$-concave functions, where $L,\lambda$ depends only on the semi-concavity of the corresponding distance functions in $X$. Since we have fixed $r>0$, these numbers $\lambda,L$ can be chosen independently of $X$.
Thus, $f_j$ can be written in the chart $O$ as the difference of two convex functions with universal Lipschitz constants $L'$.
Therefore, for any unit vector $v\in \R^n$, we have a uniform bound on the total mass of the
Radon measure $[\frac {\partial ^2 f_j} {\partial ^2v}] (\hat O)$.
This implies that all partial second derivatives of $f$ have uniformly bounded mass on $\hat O$; see also \cite[Theorem 6.8]{Evans}.

From this we deduce a uniform bound $A'$ on the total mass $[g'] (\hat O)$, for the fixed value of $r_0=\frac 1 {A^2}$.

For any $r<r_0$ we rescale the space by $\frac {r_0} r$. The total mass of the Riemannian tensor $g$ is then rescaled
by $(\frac {r_0} r)^{n-1}$. Thus, $$[g'] (\hat O) \leq A'\cdot (r_0 ) ^{1-n} \cdot r^{n-1}  .$$
We finish the proof by replacing $A$ by $\max(A, A' r_0^{1-n})$.
\end{proof}

%Let $X_{\delta}$ be the set of $\delta$-strained
%points and $S_{\delta}$ be its complement.
%Let $p\in X_{\delta}$ and let $\{(a_i,b_i)\}_{i=1}^n$ be an $(A,\delta)$ strainer near $p$. Then the distance coordinate map
%$x\mapsto (|x,a_1|,\ldots |x,a_n|)$ is locally $1+10\delta)$-biLipschitz near $p$ ~\cite{BGP} and according to ~\cite{Per-DC} the collection of such coordinate charts over all $p\in X_{\delta}$ gives $X_{\delta}$ a structure of a $DC^0$ manifold. Moreover, the averaged coordinate charts
%$$x\mapsto (\oint_{B(a_1,\eps)}|xy|dvol(y),\ldots, \oint_{B(a_n,\eps)}|xy|dvol(y))$$ where $\eps\ll A$ gives $( X_{\delta},R)$ the structure of a $DC^1$-manifold for some subset of full measure $R\subset X_{\delta}$.

Now we use \cref{cor-dc-vr} to conclude:

\begin{prop} \label{ballmeasure}
%The measure $\mu$ defined in the last section satisfies $\mu (X_{\delta } )< \infty$.
Let $C=C(n),A=A(n)$ be the constants from \pref{prop-smooth} and \lref{lem:A}.
For any point $x\in X_{r,\delta}$, any $s<r$ and any
Borel subset $K\subset B (x,r)$ the deviation measure $\mathcal{V}_s $ satisfies
$\mathcal{V}_s (K) \leq 2\cdot C\cdot A \cdot r^{n-1}.$
\end{prop}

\subsection{Decomposition in good balls}
Let the constant $A$ be as above.
A ball $B(x,r)$ in $X^n$ will be called \emph{good} if $x\in X_{r,\delta}$.
A ball $B(x,r)$ in $X$ will be called \emph{bad} if it is not good.

In this subsection we give a controlled covering result; see also Problem \ref{qe:control}.

\begin{prop}\label{prop:covering}
Let $X^n$ be an $n$-dimensional Alexandrov space without boundary.
For every compact $W\subset X$, every $\alpha >n-2$ there exists
a positive number $ q =q(W,\alpha ) >0$ with the following property. For every $x\in W$ and every $s<1 $
there exists a countable collection of good balls $B_m=B(x_m,{r_m}) \subset X$
such that
\begin{enumerate}
\item $r_m<s$ for all $m$.
\item $ \mathcal H^n \big(B (x,s) \setminus (\cup _m B_m )\big)=0$.
\item $\sum_m r_m^{\alpha}< q \cdot s^{\alpha}$.
\end{enumerate}
\end{prop}

The proof will be obtained by a recursive application of the following lemma.

\begin{lem}\label{lem:covering}
There is an integer $N =N(W,\alpha) $ with the following property. For any $p\in W$ and $\rho < 1$
the ball $B(p,{\rho})$ can be covered by at most $N$ balls
$B_i=B(x_i,{r_i})$ such that $r_i <\rho$, for all $i$, and
$$ \sum_{i\in \BAD}r_i^{\alpha}< \tfrac12\cdot \rho^{\alpha}  ,$$
where $i\in \BAD$ means that $B_i$ is a bad ball.
\end{lem}

\parit{Proof.}
Assume the contrary. Thus we can find a sequence of balls
$K_l=B(p_l,\rho_l)$
such that $p_l\in W$,
$\rho_l<1$ and
one needs at least $l$ balls to cover $K_l$, so that the conditions in the lemma are fulfilled.

Taking a subsequence we may assume that the following limit exists in the pointed Gromov--Hausdorff metric.
$$(\tfrac1{\rho_m}\cdot X,p_m)\GHto (Y,p).$$
Since the points $p_m$ range over a compact subset of $X$ and $\rho _m <1$, the
sequence is non-collapsing, i.e. $Y$ is an $n$-dimensional Alexandrov space.
By Perelman's stability theorem, $\partial Y$ is empty.
Therefore, $S:= (Y\setminus Y_{\delta} ) \cap \bar B (p,2)$ is a compact set of Hausdorff dimension $\le n-2$.

By the definition of Hausdroff dimension, we can cover $S$ by a finite number of balls
$B_i=B(x_i,{r_i})$ such that
$$\sum_ir_i^{\alpha} < (\tfrac12)^{\alpha}.$$

Any point in the remaining compact set $K\backslash (\cup_i B_i)$
is contained in $Y_{\delta}$. Therefore a small ball centered at any point of this set is good.
By compactness, we can cover $K\backslash (\cup_i B_i)$ by a finite number of good balls.
Let $N$ be the total number of balls in the obtained covering of $K$.

Lifting the constructed covering to $K_l$, for all large $l$,
we cover
the ball $K_l$ by at most $N$ balls satisfying the conditions of the lemma.
This contradiction to our assumption finishes the proof of the lemma.
\qeds

\begin{proof}[Proof of Proposition \ref{prop:covering}]
Cover $B (x,{s})$ by $N$ balls as in \lref{lem:covering} and call this covering $\mathcal F_1$. Now cover every bad
ball from the covering $\mathcal F_1$ by at most $N$ balls provided by \lref{lem:covering}. Together with the good balls from $\mathcal F_1$ the new balls define a covering $\mathcal F_2$ of $B (x,s)$. Proceeding in this way define for each natural number $k$ a covering $\mathcal F_k$ of $B(x,s)$.

Denote by $g_l^{+}$ and $g_l ^-$ the sum of $r_i ^{\alpha}$ over good, respectively bad balls $B (x_i,{r_i})$ in the covering $\mathcal F_l$.
Then, by construction, $g_{l+1} ^- < \frac 1 2 g_l ^-$ and $g_{l+1} ^+ \leq g_l ^+ + N\cdot g_l ^-$. Therefore, $g_l^- \leq 2^{-l} \cdot g_1 ^-$
and $g_l^+$ is uniformly bounded form above.
The volume of the union of bad balls in $\mathcal F_l$ is at most $g_l^-$ and converges to $0$ as $l$ goes to $\infty$.

Let $\mathcal F$ be the set of all good balls $B_j =B (x_j,{r_j})$ from all the coverings $\mathcal F_l$.
Then $\mathcal H^n (B (x,s) \setminus (\cup _{\mathcal F} B_j ) \leq \lim _{l\to \infty} g_l ^- =0$. On the other hand,
by construction,
\[\sum_{B_j\in \mathcal F} r_j^{\alpha}
= \lim _{l\to \infty} g_l^+
\leq 3{\cdot N}{\cdot} s^{\alpha}.\]

 Setting $q=3\cdot N$ finishes the proof.
\end{proof}

\subsection{Final step}
Now we can provide the
\begin{proof}[Proof of Theorem~\ref{alexandrovthm}]
Let $X$ be a fixed $n$-dimensional Alexandrov space. By the inequality of Bishop--Gromov, the deviations measures $\mathcal{V}_r$ are uniformly bounded from below
by a quadratic term in $r$.
Thus in order to control the mm-boundary we only need to bound $\mathcal{V}_r$ from above on balls in $X$.

Let the constants $A,C$ be chosen as above, so that \pref{ballmeasure} can be applied.

Let us first assume that $\partial X$ is empty.
Let $W\subset X$ be an arbitrary compact subset.
Fix $\alpha =n-\tfrac 3 2$ and choose the constant $q$ as in Proposition \ref{prop:covering}.
For any $x\in W$ and $s<1$ consider the good balls $B_i=B (x_i,{r_i})$ provided by Proposition \ref{prop:covering} and set $K'=\cup _i B_i$. Let $r<\frac 1 {A^2}$ be sufficiently small. Since $\mathcal H^n (K\setminus K')=0$
we have $ \mathcal{V}_r (K)=\mathcal{V}_r (K')$.

For all $m$ with $r<r_m$, we apply \pref{ballmeasure} and infer
$$\mathcal{V}_r(B_m \cap K)\le 2\cdot C\cdot A \cdot r\cdot r_m^{n-1} $$
On the other hand, for $r_m<r$, we have
\begin{align*}
\mathcal{V}_r (B_m\cap K)
&\leq \mathcal H^n (B_m) \leq
\\ &\leq
2 \cdot \omega _n \cdot r_m ^n <
\\ &<2\cdot \omega_n \cdot r \cdot r_m ^{n-1}
\end{align*}
Summing up and using $r_m^{n-1} < r_m ^{\alpha}$ we obtain
\begin{equation} \label{eq:density}
\begin{aligned}
\mathcal{V}_r (K)
&\leq \sum_m \mathcal{V}_r (B_m \cap K) \leq
\\&\leq
(2\cdot C\cdot A+2\cdot \omega _n) \cdot q\cdot r \cdot s^{\alpha}  .
\end{aligned}
\end{equation}

This proves that $X$ has locally finite mm-boundary.
As already mentioned and used above, any signed Radon measure $\nu$ obtained as a limit of a sequence $\mathcal{V}_{r_j} /r_j$ for some $r_j\to 0$ must be non-negative, hence a Radon measure. We fix such $\nu$.

Inequality \eqref{eq:density} implies that $\nu$ has finite $\alpha$-dimensional density at every point of $X$, in particular, $\nu$ vanishes on subsets of Hausdorff dimension $\leq n-2$.
Thus $\nu (X \setminus X_{\delta } )=0$.

Recall that $X_{\delta}$ is an almost Riemannian space. Denote, by $\mathcal N $ its minimal metric derivative measure. We extend it to a measure on all of $X$ (still denoted by $\mathcal N $) by setting it to be $0$ on $X\setminus X_{\delta}$.
By \lref{cor-mu-dc} the Radon measure $\nu$ is absolutely continuous with respect to $\mathcal N $ on compact subsets of $X_{\delta}$.
Now (1) and (3) of \tref{alexandrovthm} follow from \lref{lem:minderiv}.
This finishes the proof in the case $\partial X=\emptyset$.

Assume now that $\partial X\ne\emptyset$ and consider the doubling $Y=X\sqcup_{\partial X} X$ of $X$.
Consider $X$ as a convex subset of $Y$ and let $K\subset X$ be compact.
We find a constant $L>0$ such that for all sufficiently small $r>0$, we have
$\mathcal H^n (K\cap B(r,{2{\cdot})} (\partial X)) \leq L \cdot r$.
(This follows, for example, by the coarea formula using the Lipschitz properties of the gradient flow of the distance function $d(\cdot, \partial X)$ which is semiconcave.)
On the other hand, for $x\in K\setminus B (\partial X,r)$, the volumes of the $r$-ball in $X$ and in $Y$ coincide.
Using that $Y$ has locally finite mm-boundary, we deduce that $\mathcal{V}_r(K)$ (computed in the space $X$) is bounded from above $L\cdot r + \tilde{\mathcal{V}}_r (K)$, where $\tilde{\mathcal{V}}_r (K)$ is the deviation measure of $K$ considered as a subset of $Y$. This implies that $\mathcal{V}_r /r$ is uniformly bounded for $r\to 0$. Thus $X$ has locally finite mm-boundary as well.

Any limit of a sequence $\mathcal{V}_{r_j} /r_j$ for some $r_j\to 0$ must be again non-negative, hence a Radon measure. Outside of $\partial X$
$\nu$ coincides with the restriction of the corresponding measure defined on $Y$. From the corresponding statement about $Y$ we deduce that $\nu$ is absolutely singular with respect to $\mathcal H^n$. Moreover, $\nu$ vanishes on subsets $S\subset X\setminus \partial X$ with finite $\mathcal H^{n-1} (S)$.

It remains to prove (2), i.e. to show that the restriction of $\nu$ onto $\partial X$ is at least
$c\cdot \mathcal H^{n-1}$ for universal constant $c=c(n)$. This statement is local on $\partial X$ and needs to be verified only in small neighborhoods of points $x$ whose tangent $T_xX$
are isometric to flat halfspaces.

We fix such a point $x\in \partial X$. We further fix a sufficiently small $\varepsilon >0$ and find
a small neighborhood $U$ of $x$ in $X$ which is
$(1+\varepsilon)$-bi-Lipschitz to a half-ball in the Euclidean space. Choose an arbitrary $s>0$ such that $B (x,{2{\cdot}s})\subset U$.
Let $K= \bar B (x,s) \cap \partial X$ be the closed ball of radius $s$ in $\partial X$ with respect to the ambient metric.
Due to \cite[Section 1.6]{Evans}, it is sufficient to prove that $\nu (K) \geq c_0 \cdot s^{n-1}$ for a universal constant $c_0$ depending only on the dimension.

In order to prove this inequality, we consider any open neighborhood $V$ of $K$ in $X$.
For all small $r>0$, the neighborhood $V$ contains $B (K,{2{\cdot}r})$.
Once $\varepsilon$ has been chosen sufficiently small, the ball $B (z,r)$ in $X$ has volume at most
$ (1-k_1) \cdot \omega _n \cdot r^n$, for
any point $z\in B (K,{\frac 1 {10}{\cdot}r})$. Here $ k_1=k_1(n)>0 $ is a universal constant.
 Moreover, the set $B (K,{\frac 1 {10}{\cdot}r})$ has volume at least
 $\frac 1 {20} \cdot r \cdot \omega _{n-1} \cdot s^{n-1}$.
Integrating over $V$ (and using the inequality of Bishop--Gromov on the complement of $B (K,{\frac 1 {10}{\cdot}r})$) we deduce:
$$\mathcal{V}_r (V) \geq k_1\cdot \tfrac 1 {20} \cdot \omega _{n-1} \cdot r\cdot s^{n-1} - k_3 \cdot r^2 ,$$
for some $k_3$ depending only on the volume of $V$ and independent of $r$.
Dividing by $r$ and letting it go to $0$ we obtain $\nu (V) \geq k_4\cdot s^{n-1}$, for a universal constant $k_4>0$.
Since the neighborhood $V$ of $K$ was arbitrary,
we infer the same inequality for $K$ instead of $V$, finishing the proof.
\end{proof}

\section{Questions and Comments} \label{sec:final}
\subsection{Manifolds}
The notions of mm-boundary and mm-curvature are very easy to define but difficult to control.
For instance, the examples mentioned in the introduction require some amount of computations and estimates. On the other hand,
interesting   examples  seem to be difficult to construct as well. The first question in this direction is:

\begin{quest}
Construct a closed  manifold with a continuous Riemannian metric
that does not have finite mm-boundary.
\end{quest}

The following problem is motivated by our approach to \tref{alexandrovthm} in Sections \ref{sec-BV-estimate}, \ref{sec:Alex}.

\begin{quest} \label{qe:min}
Let $X$ be an almost Riemannian space. Can the minimal metric derivative measure be non-zero?
\end{quest}

In the language of $DC$-calculus as discussed in \cite{AB15}, this question can be reformulated as follows.  Given a compact subset $K$
on any $DC_0$-Riemannian manifold and any $\varepsilon >0$, can one cover $K$ by charts  such that the total mass of the derivative of the metric tensor
in these coordinates is bounded by $\varepsilon$?
Note that the minimal metric derivative measure must vanish if the metric can be locally defined by a Riemannian
tensor of class $W^{1,1}$, since the metric derivative measure is absolutely singular with respect to the Hausdorff measure by \lref{lem:minderiv}.

The following question is motivated by \lref{cor-mu-dc} and potential applications to geodesic flows of spaces with curvature bounded from above; see also Problem \ref{qe:CAT}.

\begin{quest} \label{qe:BV}
Let $X$ be an almost Riemannian space. Can one use the minimal metric derivative measure in order to control the deviation measures $\mathcal{V}_r$ from below?
\end{quest}

\subsection{Surfaces and hypersurfaces}
The answer to the following question is not trivial in view of Example \ref{ex:cone}.

\begin{quest}
Can one express the mm-curvature of an Alexandrov surface in terms of its curvature measure?
\end{quest}

In view of \tref{intsurface} it is reasonable to expect an affirmative answer to the following question
\begin{quest}
Do convex hypersurfaces of $\R^n$ have locally finite mm-curvature?
\end{quest}

A natural approach to this question is related to the following conjectural generalization of Bonk--Lang theorem \cite{Bonk-Lang}:
\begin{quest}
Let $X$ be a convex hypersurface sufficiently close to a flat hyperplane. Can we bound the optimal bi-Lipschitz constant for maps into the Euclidean space in terms of the total scalar curvature?
\end{quest}

Some natural generalizations of our \tref{thmconv} are possible. Probably, slightly refined arguments can be used to
prove that any DC-submanifold of a Euclidean space has vanishing mm-boundary. Using the embedding theorem of Nash, this would also provide an easy generalization of \tref{thmconv} and \tref{thmfirst} to convex hypersurfaces of smooth Riemannian manifolds.

\subsection{Alexandrov geometry and beyond}
As the next generalization of \tref{thmfirst}, one should study the case of smoothable Alexandrov spaces.

\begin{quest}
Does the mm-boundary vanish in smoothable Alexandrov spaces? Are there relations to scalar curvature measures defined in \cite{LP}?
\end{quest}

Due to the observation after Problem \ref{qe:min}, the vanishing of mm-boundary would follow from the existence of slightly smoother coordinates than the ones provided by Perelman's DC-structure.
\begin{quest}
Let $X$ be an Alexandrov space. Can one introduce coordinates on a neighborhood of the set of regular points, such that the metric is locally given by a Riemannian tensor of class $W^{1,1}$?
\end{quest}
In the two-dimensional case, the answer to this question is ``yes'' by the work of Reshetnyak~\cite{Reshetnyak-GeomIV},  see also \cite{AB16}.

Due to \tref{thmmain}, an affirmative answer to the following question should be expected. A partial answer to it has been announced by Jerome Bertrand.
\begin{quest}
Are there further connections between the size of the mm-boundary of an Alexandrov space $X$, the existence of the geodesic flow  and the ``average size'' of the cut loci of points in $X$?
\end{quest}

Should one have a chance to go beyond mm-boundary and towards mm-curvature, one would definitely need to improve the
decomposition statement \pref{prop:covering}, which provides a geometric control of the size of the set of singular points of an Alexandrov space.

\begin{quest} \label{qe:control}
Can one replace $\alpha>n-2$ by $\alpha=n-2$ in the statement of \pref{prop:covering}?
\end{quest}

An affirmative answer  has been announced by Aaron Naber.

% In the proofs of our main results about convex hypersurfaces and Alexandrov surfaces we deduced local volume bounds from
%a much stronger bound on the optimal local bi-Lipschitz constant (for maps into the Euclidean space).
It is interesting to understand if our results provide a quantitative version of bi-Lipschitz closeness of small balls to Euclidean balls.
It is known~\cite{BGP} that there exists $\kappa(n,\delta)\to 0$ as $\delta\to 0$ such that if $X=X^n$ is an Alexandrov space of curvature $\ge -1$, $x\in X$ such that $\omega _n{\cdot}r^n -\mathcal H^n (B (x,r)) \leq \delta \cdot r ^n$ then $B (x,{\frac r 4})$ is $(1+\kappa(n,\delta))$-bi-Lipschitz to a Euclidean ball.
\begin{quest}
Can $\kappa(n,\delta)$ above be chosen
%to be linear in $\delta$, i.e. to be
of the form $C(n)\cdot \delta$?
%More precisely:}
%Does there exist a constant $C=C(n)$ such that for any Alexandrov space $X=X^n$ {\color{red} of curvature $\ge -1$}, any $x\in X$ and any $\delta <\frac 1 C$ the following holds true.
%If $\omega _n{\cdot}r^n -\mathcal H^n (B (x,r)) \leq \delta \cdot r ^n$, then
%the ball $B (x,{\frac r 4})$ is $(1+C \cdot \delta )$-bi-Lipschitz to a Euclidean ball.
\end{quest}

It is natural to look at what happens for curvature bounded above:

\begin{quest} \label{qe:CAT}
Can one estimate and use the mm-boundary in geodesically complete spaces with upper curvature bounds to study the geodesic flow?
\end{quest}

Finally, it seems reasonable to expect some generalizations to spaces with Ricci curvature bounds, for instance:

\begin{quest}
Can one control the mm-boundary of noncollapsed limits of Riemannian manifolds with Ricci curvature bounded below? Can one expect something like a geodesic flow in this setting?
\end{quest}

From the work of Jeff Cheeger and Aaron Naber \cite{Ch-Na-codim4} it should follow that on any non-collapsed limit of manifolds with both-sided Ricci curvature bounds, the mm-curvature is locally finite and mm-boundary is zero.
Vanishing of the mm-boundary should then imply that the geodesic flow is defined almost everywhere and preserves the Liouville measure by the same argument as in the proof of Theorem~\ref{thmmain}.
\bibliographystyle{alpha}
%\bibliography{mmm}

\begin{thebibliography}{KMS01}





\bibitem[AB15]{AB15}
L.~Ambrosio and J.~Bertrand.
\newblock {DC} {C}alculus.
\newblock arXiv:1505.04817, 2015.



\bibitem[AB16]{AB16}
L.~Ambrosio and J.~Bertrand.
\newblock On the regularity of {Alexandrov} surfaces with curvature bounded below
\newblock {\em Analysis and Geometry in Metric Spaces}, 4,  2016.






\bibitem[AK00]{AmbKirk}
L.~ Ambrosio and B.~Kirchheim.
\newblock Rectifiable sets in metric and {B}anach spaces.
\newblock {\em Math. Ann.}, 318(3):527--555, 2000.

\bibitem[Ale16]{Alesker}
S.~Alesker.
\newblock Some conjectures on intrinsic volumes of {R}iemannian manifolds and
{Alexandrov} spaces.
\newblock arXiv:1611.09546, 2016.

\bibitem[AZ67]{AZ}
A.~D. Aleksandrov and V.~A. Zalgaller.
\newblock {\em Intrinsic geometry of surfaces}.
\newblock Translated from the Russian by J. M. Danskin. Translations of
Mathematical Monographs, Vol. 15. American Mathematical Society, Providence,
R.I., 1967.

\bibitem[Bam16]{Bamler}
R.~Bamler.
\newblock Structure theory of singular spaces.
\newblock arXiv:1603.05236, 2016.

\bibitem[BB95]{BallmannBrin}
W.~Ballmann and M.~Brin.
\newblock Orbihedra of nonpositive curvature.
\newblock {\em Inst. Hautes \'Etudes Sci. Publ. Math.}, (82):169--209 (1996),
1995.

\bibitem[BBI01]{BBI01}
D.~Burago, Yu. Burago, and S.~Ivanov.
\newblock {\em A course in metric geometry}, volume~33 of {\em Graduate Studies
in Mathematics}.
\newblock American Mathematical Society, Providence, RI, 2001.

\bibitem[Ber02]{Bernig-Alex}
A.~Bernig.
\newblock Scalar curvature of definable {A}lexandrov spaces.
\newblock {\em Adv. Geom.}, 2(1):29--55, 2002.

\bibitem[Ber03]{Bernig-CAT}
A.~Bernig.
\newblock Scalar curvature of definable {CAT}-spaces.
\newblock {\em Adv. Geom.}, 3(1):23--43, 2003.

\bibitem[BGP92]{BGP}
Yu. Burago, M.~Gromov, and G.~Perelmann.
\newblock A.{D}.{Alexandrov} spaces with curvatures bounded below.
\newblock {\em Russian Math. Surveys}, 47(2):1--58, 1992.

\bibitem[BL03]{Bonk-Lang}
M.~Bonk and U.~Lang.
\newblock Bi-Lipschitz parameterization of surfaces.
\newblock {\em Math. Ann.}, 327(1):135--169, 2003.

\bibitem[CN15]{Ch-Na-codim4}
J.~Cheeger and A.~Naber.
\newblock Regularity of {Einstein} manifolds and the codimension 4 conjecture.
\newblock {\em Annals of Mathematics}, pages 1093--1165, November 2015.

\bibitem[EG15]{Evans}
L.~C. Evans and R.~F. Gariepy.
\newblock {\em Measure theory and fine properties of functions}.
\newblock Textbooks in Mathematics. CRC Press, Boca Raton, FL, revised edition,
2015.

\bibitem[Fed59]{Fedcurvature}
H.~Federer.
\newblock Curvature measures.
\newblock {\em Trans. Amer. Math. Soc.}, 93:418--491, 1959.

\bibitem[Fed69]{Federer}
H.~Federer.
\newblock {\em Geometric measure theory}.
\newblock Grund. math. Wiss., Volume 153. Springer-Verlag, 1969.

\bibitem[GL80]{Goffmann}
C.~Goffmann and F.~Liu.
\newblock Derivative measures.
\newblock {\em Proc. Amer. Math. Soc.}, 78:218--220, 1980.

\bibitem[KMS01]{Shioya}
K.~Kuwae, Y.~Machigashira, and T.~Shioya.
\newblock Sobolev spaces, {L}aplacian, and heat kernel on {A}lexandrov spaces.
\newblock {\em Math. Z.}, 238(2):269--316, 2001.

\bibitem[LP17]{LP}
N.~Lebedeva and A.~Petrunin.
\newblock Curvature tensors on {Alexandrov} spaces.
\newblock {\em In preparation}, 2017.

\bibitem[OS94]{Otsu-Shioya}
Y.~Otsu and T.~Shioya.
\newblock The {R}iemannian structure of {A}lexandrov spaces.
\newblock {\em J. Differential Geom.}, 39(3):629--658, 1994.

\bibitem[Per91]{P2}
G.~Perelman.
\newblock {A. D. Alexandrov} spaces with curvature bounded below {II}.
\newblock {\em preprint}, 1991.

\bibitem[Per95]{Per-DC}
G.~Perelman.
\newblock {DC} structure on {A}lexandrov space with curvature bounded below.
\newblock preprint, http://www.math.psu.edu/petrunin/papers/papers.html, 1995.

\bibitem[Pet98]{Petparallel}
A.~Petrunin.
\newblock Parallel transportation for {A}lexandrov space with curvature bounded
below.
\newblock {\em Geom. Funct. Anal.}, 8(1):123--148, 1998.

\bibitem[Pet07]{Petsemi}
A.~Petrunin.
\newblock Semiconcave functions in {A}lexandrov's geometry.
\newblock In {\em Surveys in differential geometry. {V}ol. {XI}}, volume~11 of
{\em Surv. Differ. Geom.}, pages 137--201. Int. Press, Somerville, MA, 2007.

\bibitem[PP96]{PP}
G.~Perelman and A.~Petrunin.
\newblock Quasigeodesics and gradient curves in alexandrov spaces.
\newblock preprint, http://www.math.psu.edu/petrunin/papers/papers.html, 1996.

\bibitem[Res93]{Reshetnyak-GeomIV}
Yu.~G. Reshetnyak.
\newblock Two-dimensional manifolds of bounded curvature.
\newblock In {\em Geometry, {IV}}, volume~70 of {\em Encyclopaedia Math. Sci.},
pages 3--163, 245--250. Springer, Berlin, 1993.

\bibitem[Sch93]{Schneider}
R.~Schneider.
\newblock Convex surfaces, curvature and surface area measures.
\newblock In {\em Handbook of convex geometry, {V}ol.\ {A}, {B}}, pages
273--299. North-Holland, Amsterdam, 1993.

\bibitem[Zaj79]{zajicek}
Zaj\'i\v{c}ek, Lud\v{e}k
On the differentiation of convex functions in finite and infinite dimensional spaces.
Czechoslovak Math. J. 29(104) (1979), no. 3, 340--348.

\bibitem[Zam82]{Zam-inv}
T.~Zamfirescu.
\newblock Many endpoints and few interior points of geodesics.
\newblock {\em Invent. Math.}, 69(2):253--257, 1982.

\bibitem[Zam92]{Zam-quest}
T.~Zamfirescu.
\newblock Long geodesics on convex surfaces.
\newblock {\em Math. Ann.}, 293(1):109--114, 1992.

\bibitem[Zie89]{Ziemer}
W.~P. Ziemer.
\newblock {\em Weakly differentiable functions}, volume 120 of {\em Graduate
Texts in Mathematics}.
\newblock Springer-Verlag, New York, 1989.
\newblock Sobolev spaces and functions of bounded variation.

\end{thebibliography}
%\end{document}

\end{document}